\documentclass[a4paper,11pt]{amsart}

\usepackage[top=3cm,bottom=3cm,left=3cm,right=3cm]{geometry}

\usepackage{enumerate}

\usepackage{mathrsfs}
\usepackage{amsfonts}
\usepackage{amsmath}
\usepackage{amssymb}
\usepackage{amsthm}
\usepackage{stmaryrd}
\usepackage{textcomp}
\usepackage{mathtext}
\usepackage{yfonts}

\usepackage{amscd}
\usepackage{tikz}
\usepackage{tikz-cd}
\usepackage{sidecap}
\usetikzlibrary{matrix,arrows}

\usepackage{mwe} \usepackage{subcaption}

\usepackage{float}
\floatstyle{boxed}
\restylefloat{table}
\restylefloat{figure}

\newtheoremstyle{break}{9pt}{9pt}{}{}{\bfseries}{.}{\newline}{}
\theoremstyle{break}

\makeatletter
\newcommand\suchthat{\@ifstar
  {\mathrel{}\middle\vert\mathrel{}}
  {\mid}}
\makeatother

\newtheorem{theorem}{Theorem}[section]
\newtheorem{lemma}[theorem]{Lemma}

\newtheorem{remark}[theorem]{Remark}

\newcommand{\eps}{{\epsilon}}

\newcommand{\equivalent}{ \Longleftrightarrow }

\newcommand{\Trace}{\operatorname{Tr}}

\newcommand{\trace}{\operatorname{tr}}

\newcommand{\Ext}{\operatorname{Ext}}
\newcommand{\ext}{\operatorname{ext}}

\newcommand{\cartan}{{\mathsf d}}
\newcommand{\cartanlambda}{{\mathsf d}\lambda}
\newcommand{\nablalambda}{\nabla\lambda}

\newcommand{\curl}{\operatorname{curl}}
\newcommand{\divergence}{\operatorname{div}}

\DeclareMathOperator*{\linhull}{span}

\newcommand{\bbC}{{\mathbb C}}

\newcommand{\bbN}{{\mathbb N}}

\newcommand{\bbZ}{{\mathbb Z}}

\newcommand{\calB}{{\mathcal B}}

\newcommand{\calP}{{\mathcal P}}

\newcommand{\calS}{{\mathcal S}}
\newcommand{\calT}{{\mathcal T}}

\newcommand{\calV}{{\mathcal V}}

\newcommand{\springerqed}{}
\begin{document}

\title
[Bases FEEC]
{On Basis Constructions in\\ Finite Element Exterior Calculus}

\author{Martin W. Licht}

\address{\'Ecole Polytechnique F\'ed\'erale de Lausanne (EPFL), 1015 Lausanne, Switzerland}

\email{martin.licht@epfl.ch}

\thanks{This research was supported by the European Research Council through 
the FP7-IDEAS-ERC Starting Grant scheme, project 278011 STUCCOFIELDS.
This research was supported in part by NSF DMS/RTG Award 1345013 and DMS/CM Award 1262982.
The author would like to thank the Isaac Newton Institute for Mathematical Sciences, Cambridge, 
for support and hospitality during the programme "Geometry, compatibility and structure preservation in computational differential equations" 
where work on this paper was undertaken. This work was supported by EPSRC grant no EP/K032208/1.
Parts of this article are based on the author's PhD Thesis.}



\begin{abstract}
  We give a systematic self-contained exposition of
  how to construct geometrically decomposed bases and degrees of freedom
  in finite element exterior calculus. 
In particular, we elaborate upon a previously overlooked basis 
  for one of the families of finite element spaces, 
  which is of interest for implementations. 
Moreover, 
  we give details for the construction of isomorphisms and duality pairings 
  between finite element spaces.
  These structural results show, for example, 
  how to transfer linear dependencies between canonical spanning sets,
  or how to derive the degrees of freedom. 
\end{abstract}

\maketitle

\section{Introduction}
\label{sec:introduction}

\emph{Exterior calculus} is a canonical approach towards mathematical electromagnetism.
Utilizing exterior calculus in numerical analysis hence seems to be a natural choice. 
\emph{Finite element exterior calculus} (FEEC \cite{AFW2}) formalizes several well-known finite element methods 
in the language of exterior calculus. 
A particular achievement of FEEC has been the classification of \emph{finite element de~Rham complexes}.
Research in finite element exterior calculus has provided a unified framework for 
bases, degrees of freedom, and geometric decompositions for classical vector-valued finite element spaces
(see \cite{AFW1,AFWgeodecomp,hiptmair2001higher,hiptmair2002finite,rapetti2007geometrical,rapetti2009whitney}, for example).

This exposition addresses the construction of spanning sets, bases, and degrees of freedom in finite element exterior calculus. 
We give a systematic and comprehensive account of results which previously had remained distributed 
over different sources in the literature.
While that serves an expository purpose, 
our presentation leads to new algebraic insights into finite element exterior calculus. 
Recall that the major difficulty in the theory of vector-valued finite element spaces,
such as Brezzi-Douglas-Marini spaces, Raviart-Thomas spaces, and N\'ed\'elec spaces,
is that these spaces have no canonical bases but merely canonical spanning sets. 
This is a critical difference to scalar-valued finite element spaces. 
We contribute a full description of the linear dependencies within the canonical spanning sets. 
This allows an explicit description of the canonical isomorphisms and duality pairings 
in finite element exterior calculus. 
Our exposition furthermore reveals a simple but previously overlooked basis for one family of finite element differential forms. 
\\

Quite remarkably, 
finding explicit bases for the Brezzi-Douglas-Marini space and the N\'ed\'elec spaces of the second kind 
seems to have been an open problem for quite some time even though these spaces have been known long since. 
The original articles by Brezzi, Douglas, and Marini \cite{brezzi1985two} and N\'ed\'elec \cite{nedelec1986new}
describe the degrees of freedom, 
but it seems that explicit bases for any polynomial degree have appeared in the literature only twenty years later:
we point out the contributions by 
Arnold, Falk, and Winther \cite{AFWgeodecomp}, Ervin \cite{ervin2012computational}, and Bentley \cite{bentley2017explicit}.  
Explicit bases for the Raviart-Thomas spaces and the N\'ed\'elec spaces of the first kind 
have appeared around the same time (e.g.\ \cite{gopalakrishnan2005nedelec}). 
We describe examples for our bases later in this article (see Section~\ref{sec:finiteelementspaces}) 
after having established the notation.

Finding geometrically decomposed bases and degrees of freedom 
is much more challenging for vector-valued than for scalar-valued finite element spaces. 
The major difference is that the canonical spanning sets of the former are linearly dependent
while the ones of the latter are linearly independent. 
This exposition establishes explicit formulas 
for finite element basis forms in barycentric coordinates that are simple and readily implementable. 
Our bases for the spaces of higher order Whitney forms $\calP_{r}^{-}\Lambda^{k}(T)$ and $\mathring\calP^{-}_{r}\Lambda^{k}(T)$
coincide with the bases in \cite{AFWgeodecomp}. 
By contrast, 
we propose bases for the spaces $\calP_{r}^{}\Lambda^{k}(T)$ and $\mathring\calP_{r}^{}\Lambda^{k}(T)$ that have not yet been stated in the literature; 
in particular, they are different from the bases in \cite{AFWgeodecomp}. 
These bases seems to have gone unnoticed until now but are of appealing simplicity.

The bases for the finite element spaces in this article are subsets of the canonical spanning sets. 
We contribute a complete classification of the linear dependencies in those canonical spanning sets.
This improves our understanding of two concepts of finite element exterior calculus:
\emph{canonical isomorphisms} and \emph{duality pairings} between finite element spaces of differential forms. 
Whenever $T$ is an $n$-dimensional simplex 
and $k$ and $r$ are non-negative integers, we have isomorphisms 
\begin{gather}
 \label{math:introductionofisomorphisms}
 \calP_r\Lambda^{k}(T)
 \simeq
 \mathring\calP^{-}_{r+k+1}\Lambda^{n-k}(T)
 ,
 \quad
 \calP^{-}_{r+1}\Lambda^{k}(T)
 \simeq
 \mathring\calP_{r+k+1}\Lambda^{n-k}(T)
 .
\end{gather} 
Here, the spaces on the right-hand sides have vanishing trace along the simplex boundary. 
Moreover, 
to each of these two isomorphic pairs corresponds a non-degenerate duality pairing,
which is induced by the integral pairing over $T$. 
The seminal publication by Arnold, Falk, and Winther \cite{AFW1} introduces those isomorphisms and duality pairings. 
Subsequent work by Christiansen and Rapetti \cite{christiansen2013high} 
states that the first isomorphism of \eqref{math:introductionofisomorphisms} 
preserves the canonical spanning sets. 
Building upon their contribution, 
we use new proof techniques to show that not only the first but also the second isomorphisms preserve the canonical spanning sets. 
Consequently, these isomorphisms translate linear dependencies and independencies between these spaces. 
This leads to an explicit description of the duality pairings between finite element spaces 
and a new, straightforward way of deriving the degrees of freedom in finite element exterior calculus.

While much research on vector-valued finite element spaces has addressed 
aspects such as condition numbers, sparsity properties, hierarchical structures, or fast evaluation 
\cite{ainsworth2003hierarchic,schoberl2005high,ainsworth2011bernstein,beuchler2012sparsity,beuchler2013sparsity,kirby2014low,kirby2018general},
we address purely algebraic and combinatorial aspects. 
For the purpose of overview, 
we recapitulate previous works on basis constructions in finite element exterior calculus from the literature. 
The seminal exposition by Arnold, Falk, and Winther \cite[Chapter 4]{AFW1}
begins by devising a basis for $\calP_{r}^{-}\Lambda^{k}(T)$,
then determines a geometrically decomposed basis of the dual space $\calP_{r}\Lambda^{k}(T)^{\ast}$,
and subsequently a geometrically decomposed basis of the dual space $\calP_{r}^{-}\Lambda^{k}(T)^{\ast}$.
After outlining bases for spaces with vanishing trace,
they find geometrically decomposed bases for $\calP_{r}^{-}\Lambda^{k}(T)$
and, implicitly, for $\calP_{r}\Lambda^{k}(T)$. 
In a successive article \cite{AFWgeodecomp},
we are given bases for the spaces with vanishing trace,
$\mathring\calP^{-}_{r}\Lambda^{k}(T)$ and $\mathring\calP_{r}\Lambda^{k}(T)$.
The latter publication studies extension operators 
and geometrically decomposed bases but
their basis of $\calP_{r}\Lambda^{k}(T)$ in \cite{AFWgeodecomp}
is generally different from the one in \cite{AFW1}. 
Our present exposition develops the theory of bases and degrees of freedom in a different manner,
adapting techniques from prior research on finite element differential forms 
\cite{AFW1,AFWgeodecomp,christiansen2013high,hiptmair2001higher,rapetti2007geometrical,rapetti2009whitney}. 
Most importantly, we directly construct geometrically decomposed bases 
for the two families of finite element spaces independently from each other. 
In particular, our basis for $\calP_{r}\Lambda^{k}(T)$
contains the basis of $\mathring\calP_{r}\Lambda^{k}(T)$ as a subset. 
We emphasize that we develop bases 
without referring to any degrees of freedom in the first place.

The remainder of this work is structured as follows.
In Section~\ref{sec:combinatorics}, 
we review combinatorial results, exterior calculus, and polynomial differential forms.
Section~\ref{sec:auxiliarylemmas} summarizes some auxiliary lemmas. 
In Section~\ref{sec:finiteelementspaces}, we introduce spaces of polynomial differential forms, 
construct geometrically decomposed bases, and define extension operators. 
Subsequently, we study the isomorphy relations in Section~\ref{sec:lineardependencies} 
and the duality pairings in Section~\ref{sec:duality}.
We supplement applications to finite element spaces over triangulations in Section~\ref{sec:geometricdecompositions}.

\section{Notation and Definitions}
\label{sec:combinatorics}

We introduce and discuss notions regarding combinatorics and differential forms over simplices. 
All vector spaces in this publication are over the complex numbers unless noted otherwise;
we write $\overline z \in \bbC$ for the complex conjugate of $z \in \bbC$. 

\subsection{Combinatorics}
\label{subsec:combinatorics}

We write $[m:n] = \{m, \dots, n\}$ for $m, n \in \bbZ$. 
Note that this set may also be empty.
For $m, n \in \bbZ$ with $m \neq n$
we let $\eps(m,n) = 1$ if $m < n$ and $\eps(m,n) = -1$ if $m > n$. 

For any mapping $\alpha \colon [0:n] \rightarrow \bbN_{0}$ 
we write $\lvert\alpha\rvert := \sum_{i=0}^{n} \alpha(i)$. 
Given numbers $r, n \in \bbN_{0}$, 
we let $A(r,n)$ be the set of all mappings $\alpha \colon [0:n] \rightarrow \bbN_{0}$
for which $\lvert\alpha\rvert = r$.
The set $A(r,n)$ is also known as the set of \emph{multiindices} of degree $r$ over the set $[0:n]$. 
In this article, $r$ will usually be the polynomial degree of a finite element space.
The sum $\alpha+\beta$ of $\alpha \in A(r,n)$ and $\beta \in A(s,n)$
is defined in the obvious manner as a member of $A(r+s,n)$.
Whenever $\alpha \in A(r,n)$, we write
\begin{align}
 \label{math:bracketofmultiindex}
 [\alpha] := \left\{\; i \in [0:n] \suchthat \alpha(i) > 0 \;\right\},
\end{align}
and we write $\lfloor\alpha\rfloor$ for the minimal element of $[\alpha]$ 
provided that $[\alpha]$ is not empty, and $\lfloor\alpha\rfloor = \infty$ otherwise.
When $\alpha \in A(r,n)$ and $p \in [0:n]$,
then $\alpha + p$ denotes the unique member of $A(r+1,n)$ satisfying $(\alpha+p)(p) = \alpha(p)+1$
and coinciding with $\alpha$ otherwise;
similarly, when $p \in [\alpha]$, then 
$\alpha - p$ denotes the unique member of $A(r-1,n)$ satisfying $(\alpha-p)(p) = \alpha(p)-1$
and coinciding with $\alpha$ otherwise. 

For $a,b,m,n \in \bbN_{0}$,
we let $\Sigma(a:b,m:n)$ be the set of strictly ascending mappings from $[a:b]$ to $[m:n]$.
We call those mappings also \emph{alternator indices}. 
We write $\Sigma(a:b,m:n) := \{\emptyset\}$ whenever $a > b$.
For any $\sigma \in \Sigma(a:b,m:n)$ we let 
\begin{align}
 \label{math:bracketofsigma}
 [\sigma] := \left\{\;  \sigma(i) \suchthat i \in [a:b] \; \right\},
\end{align}
and we write $\lfloor\sigma\rfloor$ for the minimal element of $[\sigma]$  
provided that $[\sigma]$ is not empty, and $\lfloor\sigma\rfloor = \infty$ otherwise.
Furthermore, if $q \in [m:n] \setminus [\sigma]$, then we write $\sigma + q$
for the unique element of $\Sigma(a:{b+1},m:n)$ with image $[\sigma]\cup\{q\}$.
In that case, we also write $\eps(q,\sigma)$ for the sign of the permutation 
that orders the sequence $q, \sigma(a), \dots, \sigma(b)$ in ascending order,
and we write $\eps(\sigma,q)$ for the sign of the permutation 
that orders the sequence $\sigma(a), \dots, \sigma(b), q$ in ascending order.
Similarly, if $p \in [\sigma]$, then we write $\sigma - p$ 
for the unique element of $\Sigma(a:b-1,m:n)$ with image $[\sigma]\setminus\{p\}$.
Note that $\eps(\sigma,q) = (-1)^{b-a+1} \eps(q,\sigma)$. 

We abbreviate $\Sigma(k,n) = \Sigma(1:k,0:n)$
and $\Sigma_{0}(k,n) = \Sigma(0:k,0:n)$. 
If $n$ is understood and $k, l \in [0:n]$,
then for any $\sigma \in \Sigma(k,n)$
we define $\sigma^c \in \Sigma_{0}(n-k,n)$
by the condition $[\sigma]\cup[\sigma^c] = [0:n]$,
and for any $\rho \in \Sigma_{0}(l,n)$
we define $\rho^c \in \Sigma(n-l,n)$
by the condition $[\rho]\cup[\rho^c] = [0:n]$.
In particular, $\sigma^{cc} = \sigma$ and $\rho^{cc} = \rho$.
Note that $\sigma^c$ and $\rho^c$ implicitly depend on $n$,
the value of which will always be clear from context. 

When $\sigma \in \Sigma(k,n)$ and $\rho \in \Sigma_{0}(l,n)$
with $[\sigma]\cap[\rho] = \emptyset$,
then $\eps(\sigma,\rho)$ denotes the sign of the permutation ordering the sequence
$\sigma(1),\dots,\sigma(k),\rho(0),\dots,\rho(l)$ in ascending order, and we let
$\sigma+\rho \in \Sigma(0:k+l,0:n)$
be the unique strictly ascending mapping from $[0:k+l]$ to $[0:n]$
whose image is the set $[\sigma]\cup[\rho]$.

\begin{remark} 
    Several times in this article we employ combinatorial identities 
    that relate signs of permutations. Intuitively, these often follow
    from different procedures of ordering some finite sequence of numbers
    $a_1, a_2, \dots, a_N$ into ascending order.
    It is most helpful to keep in mind 
    that the sign of a permutation is precisely the parity of the number of 
    \emph{transpositions of adjacent slots} that order the sequence.       
\end{remark}

\subsection{Simplices}
\label{subsec:simplices}

Let $n \in \bbN_{0}$. 
An $n$-dimensional simplex $T$ is the convex closure 
of pairwise distinct points $v^T_0, \dots, v^T_n$ in Euclidean space,
called the \emph{vertices} of $T$,
such that the vertices are an affinely independent set. 
Note that the dimension of the ambient Euclidean space 
must be at least $n$ for any simplex to exist, but otherwise it does not matter.
We call $F \subseteq T$ a \emph{subsimplex} of $T$
if the set of vertices of $F$ is a subset of the set of vertices of $T$. 

An \emph{ordered simplex} is a simplex with an ordering of its set of vertices (see \cite{FuchsViro}).
We henceforth assume that all simplices in this article are ordered. 
Without loss of generality, we assume moreover that subsimplices order their vertices in the same manner as those vertices are ordered in their parent simplices.

Suppose that $F$ is an $m$-dimensional subsimplex of $T$
with ordered vertices $v^F_0, \dots, v^F_m$. 
We write $\imath(F,T) \colon F \rightarrow T$ for the set inclusion of $F$ into $T$.
With a mild abuse of notation, we let $\imath(F,T) \in \Sigma_{0}(m,n)$
denote the unique mapping that satisfies $v^{T}_{\imath(F,T)(i)} = v^{F}_{i}$.
Loosely speaking, this mapping assigns to each vertex index of $F$
the corresponding vertex index of $T$, and the mapping is ascending 
because we assume that $F$ orders its vertices in the same manner as $T$ does.

\subsection{Barycentric Coordinates and Differential Forms}
\label{subsec:barycentriccoordinates}

Let $T$ be a simplex of dimension $n$. 
Following the notation of \cite{AFW1},
we write $\Lambda^{k}(T)$ for the space of \emph{differential $k$-forms} over $T$ 
with smooth bounded coefficients of all orders. Recall that a differential form is a function which takes values 
in the $k$-th exterior power 
of the dual of the tangential space of the simplex $T$. 
In the case $k=0$, the space $\Lambda^{0}(T) = C^{\infty}(T)$
is just the space of smooth functions over $T$ 
with uniformly bounded derivatives. 
Furthermore, $\Lambda^{k}(T)$ is the trivial vector space unless $0 \leq k \leq n$. 

We recall the \emph{exterior product} $\omega \wedge \eta \in \Lambda^{k+l}(T)$
for $\omega \in \Lambda^{k}(T)$ and $\eta \in \Lambda^{l}(T)$
and that it satisfies $\omega \wedge \eta = (-1)^{kl} \eta \wedge \omega$. 
We let $\cartan \colon \Lambda^{k}(T) \rightarrow \Lambda^{k+1}(T)$
denote the \emph{exterior derivative}. 
It satisfies $\cartan\left( \omega \wedge \eta \right) = \cartan\omega \wedge \eta + (-1)^{k} \omega \wedge \cartan\eta$
for $\omega \in \Lambda^{k}(T)$ and $\eta \in \Lambda^{l}(T)$. 
We also recall that the integral $\int_{T} \omega$ of a differential $n$-form over $T$ is well-defined.
The sign of that integral is determined by the orientation of the simplex. 
We refer to \cite{AFW1} and \cite{LeeSmooth} for more background. 
In this article, we focus on a special class of differential forms,
namely \emph{barycentric differential forms}. 
\\

The \emph{barycentric coordinates} $\lambda^T_0, \lambda^T_1, \dots, \lambda^T_n \in \Lambda^{0}(T)$
are the unique affine functions over $T$ that satisfy the \emph{Lagrange property}
\begin{align}
 \label{math:lagrangeproperty}
 \lambda^T_{i}( v_j ) = \delta_{ij},
 \quad
 i,j \in [0:n].
\end{align}
The barycentric coordinate functions of $T$ are linearly independent 
and constitute a partition of unity:
\begin{align}
 \label{math:partitionofunity}
 1 = \lambda^T_0 + \lambda^T_1 + \dots + \lambda^T_n
 .
\end{align}
We write $\cartanlambda_0^{T}, \cartanlambda_1^{T}, \dots, \cartanlambda_n^{T} \in \Lambda^{1}(T)$ 
for the exterior derivatives of the barycentric coordinates,
which are differential $1$-forms and constitute a partition of zero:  
\begin{align}
 \label{math:partitionofzero}
 0 = \cartanlambda^T_0 + \cartanlambda^T_1 + \dots + \cartanlambda^T_n
 .
\end{align}
Up to scaling, this is the only linear dependence
between the exterior derivatives of the barycentric coordinate functions.
\\

We consider several classes of differential forms over $T$
that are expressed in terms of the barycentric polynomials and their exterior derivatives.
When $r \in \bbN_{0}$ and $\alpha \in A(r,n)$, 
then the corresponding \emph{barycentric polynomial} over $T$ is 
\begin{align}
 \label{math:definitionbarycentricpolynomial}
 \lambda_T^{\alpha}
 &:=
 \prod_{ i=0 }^{n} (\lambda^T_i)^{\alpha(i)}
 .
\end{align}
When $a,b \in \bbN_{0}$ and $\sigma \in \Sigma(a:b,0:n)$,
then the corresponding \emph{barycentric alternator} is 
\begin{align}
 \label{math:definitionbarycentricdifferentialform}
 \cartanlambda^T_{\sigma} 
 :=
 \cartanlambda^T_{\sigma(a)} \wedge\dots\wedge \cartanlambda^T_{\sigma(b)}
 .
\end{align}
Here, by definition, $\cartanlambda^{T}_{\emptyset} := 1$ in the special case $\sigma = \emptyset$.
Also, we occasionally write
\begin{align}
    \label{math:alternatorpolynomial} 
    \lambda_{\sigma}^{T} = \lambda_{\sigma(a)}^{T} \cdots \lambda_{\sigma(b)}^{T}
    .
\end{align}
Finally, 
whenever $a,b \in \bbN_{0}$ and $\rho \in \Sigma(a:b,0:n)$,
then the corresponding \emph{Whitney form} is 
\begin{align}
 \label{math:definitionwhitneyform}
 \phi^{T}_{\rho}
 :=
 \sum_{p \in [\rho]} \eps(p,\rho-p) \lambda^{T}_{p} \cartanlambda^{T}_{\rho-p}
 .
\end{align}
In the special case 
when $\rho_{T} \colon [0:n] \rightarrow [0:n]$ is the single member of $\Sigma_{0}(n,n)$,
we then write $\phi_{T} := \phi_{\rho_{T}}$ for the associated Whitney form.

In what follows, the differential forms \eqref{math:definitionbarycentricpolynomial},
\eqref{math:definitionbarycentricdifferentialform}, \eqref{math:definitionwhitneyform}, their sums, and their exterior products are called \emph{barycentric differential forms} over $T$.

\begin{remark}
  Whenever a fixed simplex $T$ is understood and there is no danger of ambiguity,
  we may simplify the notation by writing 
  \begin{align*}
    \lambda_i \equiv \lambda^T_i,
    \quad
    \lambda^{\alpha} \equiv \lambda_T^\alpha,
    \quad
    \cartanlambda_{\sigma} \equiv \cartanlambda_\sigma^T,
    \quad
    \lambda_{\sigma} \equiv \lambda_\sigma^T,
    \quad 
    \phi_{\rho} \equiv \phi_\rho^T
    .
  \end{align*}
  This notational convention is inspired by \cite{AFWgeodecomp} and \cite{christiansen2013high}.
  With our choice of notation, the simplex $T$ is always a superscript 
  except for the barycentric monomials.
\end{remark}

\subsection{Traces}
\label{subsec:traces}

Let $T$ be an $n$-dimensional simplex and let $F \subseteq T$ be a subsimplex of $T$ of dimension $m$.
The inclusion $\imath(F,T) \colon F \rightarrow T$ introduced above 
naturally induces a mapping $\trace_{T,F} \colon \Lambda^{k}(T) \rightarrow \Lambda^{k}(F)$
by taking the pullback.
We call $\trace_{T,F}$ the \emph{trace} from $T$ onto $F$.
It is well-known that $\cartan \trace_{T,F} \omega = \trace_{T,F} \cartan \omega$
for all $\omega \in \Lambda^{k}(T)$, that is, 
the exterior derivative commutes with taking traces. 
In the case of $0$-forms, the trace is just the natural restriction of functions,
and in the case of $1$-forms, we may think of it as a tangential trace. 

Taking into account the ordering of the vertices, 
we get explicit formulas for the traces of barycentric differential forms.
Write $[\imath(F,T)]$ for the set of indices of those vertices of $T$
that are also vertices of $F$.
We let 
$\imath(F,T)^{\dagger} \colon [\imath(F,T)] \rightarrow [0:m]$
be the inverse of the mapping 
$\imath(F,T) \colon [0:m] \rightarrow [\imath(F,T)]$.
\\

Consider $i \in [0:n]$. 
If $i \notin [\imath(F,T)]$, then $v_{i}^{T}$ is a vertex of $T$ that is not a vertex of $F$,
and in that case we have $\trace_{T,F} \lambda^T_i = 0$.
If instead $i \in [\imath(F,T)]$,
then there exists $j \in [0:m]$ such that $i = \imath(F,T)(j)$,
and in that case we have $\trace_{T,F} \lambda^T_i = \lambda_j^F$
or, equivalently, $\trace_{T,F} \lambda^T_i = \lambda_{\imath(F,T)^{\dagger}(i)}^F$. 
Analogous observations follow for the exterior derivatives of the barycentric coordinates.

Let $\alpha \in A(r,n)$ be a multiindex. 
If $[\alpha] \nsubseteq [\imath(F,T)]$, 
then $\trace_{T,F} \lambda_T^\alpha = 0$.
If instead $[\alpha] \subseteq [\imath(F,T)]$,
then there exists $\widehat\alpha \in A(r,m)$ with
$\widehat\alpha = \alpha \circ \imath(F,T)$,
and we have 
\begin{align}
 \trace_{T,F} \lambda_T^{\alpha} = \lambda_F^{\widehat\alpha}
 .
\end{align}
Let $\sigma \in \Sigma(a:b,0:n)$ be an alternator index. 
If $[\sigma] \nsubseteq [\imath(F,T)]$, 
then we have $\trace_{T,F} \cartanlambda^T_\sigma = 0$.
If instead $[\sigma] \subseteq [\imath(F,T)]$,
then there exists $\widehat\sigma \in \Sigma(a:b,0:m)$ with
$\imath(F,T) \circ \widehat\sigma = \sigma$,
or equivalently, $\widehat\sigma = \imath(F,T)^{\dagger} \circ \sigma$,
and we then have 
\begin{align}
 \trace_{T,F} \cartanlambda_{\sigma}^{T} = \cartanlambda_{\widehat\sigma}^{F},
 \quad 
 \trace_{T,F} \phi_{\sigma}^{T} = \phi_{\widehat\sigma}^{F}. 
\end{align}

\section{Auxiliary Lemmas}
\label{sec:auxiliarylemmas}

In this section we provide some auxiliary results 
regarding barycentric differential forms over an $n$-dimensional simplex $T$.

\begin{lemma} \label{prop:barycentricdifferentialscombinatorics}
If $\sigma \in \Sigma(a:b,m:n)$ and $p \in [\sigma]$, 
 then
 \begin{align}
  \label{math:barycentricdifferentialscombinatorics}
  \cartanlambda_{\sigma} = \eps( p, \sigma - p ) \cartanlambda_{p} \wedge \cartanlambda_{\sigma - p}.
 \end{align}
\end{lemma}
\begin{proof}
 This is a simple calculation.
\springerqed\end{proof}

\begin{lemma} \label{prop:whitneyformrecursion}
 Let $k \in [0:n]$.
 If $\rho \in \Sigma_{0}(k,n)$ and $q \in [0:n]$ with $q \notin [\rho]$,
 then
 \begin{align}
  \label{math:whitneyformrecursion}
  \eps(q,\rho) \phi_{\rho+q}
  =
  \lambda_{q} \cartanlambda_{\rho} 
  -
  \cartanlambda_q \wedge \phi_\rho
.
 \end{align}
\end{lemma}
\begin{proof}
  Letting $\rho$ and $q$ be as in the statement of the lemma,
  we use the definition of Whitney forms \eqref{math:definitionwhitneyform}
  and \eqref{math:barycentricdifferentialscombinatorics}
  to find 
  \begin{align*}
    \phi_{\rho+q}
    &
    =
    \sum_{ l \in [\rho+q] }
    \eps(l,\rho+q-l) \lambda_{l} \cartanlambda_{\rho+q-l}
    \\&
    =
    \eps(q,\rho) \lambda_{q} \cartanlambda_{\rho}
    +
    \sum_{ l \in [\rho] }
    \eps(l,\rho+q-l) \lambda_{l} \cartanlambda_{\rho+q-l}
    \\&
    =
    \eps(q,\rho) \lambda_{q} \cartanlambda_{\rho}
    +
    \cartanlambda_q \wedge 
    \sum_{ l \in [\rho] }
    \eps(l,\rho+q-l) \eps(q,\rho-l) \lambda_{l} \cartanlambda_{\rho-l}
    .
  \end{align*}
  For any $l \in [\rho]$ one finds that
  \begin{gather*}
    \eps(l,\rho+q-l) = \eps(l,q) \eps(l,\rho-l)
    , \quad 
    \eps(q,\rho-l) = \eps(q,l) \eps(q,\rho)
    , \quad 
    \eps(l,q) = -\eps(q,l)
    .
  \end{gather*}
  This observation leads to 
  \begin{align*}
    \phi_{\rho+q}
    &
    =
    \eps(q,\rho) \lambda_{q} \cartanlambda_{\rho}
    -
    \cartanlambda_q \wedge 
    \sum_{ l \in [\rho] }
    \eps(l,\rho-l) \eps(q,\rho) \lambda_{l} \cartanlambda_{\rho-l}
    \\
    &
    =
    \eps(q,\rho) \lambda_{q} \cartanlambda_{\rho}
    -
    \cartanlambda_q \wedge 
    \eps(q,\rho) 
    \phi_{\rho}
    .
  \end{align*}
  We multiply both sides with $\eps(q,\rho)$, and the desired result follows. 
\springerqed\end{proof}

Whenever $k \in [0:n]$ and $\rho \in \Sigma_{0}(k,n)$,
then definitions and Lemma~\ref{prop:barycentricdifferentialscombinatorics} 
show the differential of the corresponding Whitney form is 
\begin{align}
 \label{math:whitneyformdifferential}
 \cartan\phi_{\rho} = (k+1) \cartanlambda_{\rho}
 .
\end{align}
A converse to that is the next result, 
which has appeared as Proposition 3.4 in \cite{christiansen2013high},
and also as Equation (6.6) in \cite{AFWgeodecomp}.

\begin{lemma} \label{prop:differentialdecomposition}
 Let $k \in [0:n]$ and $\rho \in \Sigma_{0}(k,n)$.
 Then 
 \begin{align}
  \label{math:differentialdecomposition}
  \cartanlambda_\rho
  =
  \sum_{ q \in [\rho^c] } \eps(q,\rho) \phi_{\rho+q}.
 \end{align}
\end{lemma}
\begin{proof}
 We use Lemma~\ref{prop:whitneyformrecursion} and see
 \begin{align*}
  \sum_{q \in [\rho^c]}
  \eps(q,\rho) \phi_{\rho+q}
  &=
\sum_{q \in [\rho^c]} \lambda_{q} \cartanlambda_{\rho} 
  -
  \sum_{q \in [\rho^c] } \cartanlambda_q \wedge \phi_\rho.
\end{align*}
 Using \eqref{math:partitionofzero}, \eqref{math:definitionwhitneyform},
 \eqref{math:barycentricdifferentialscombinatorics}, and \eqref{math:partitionofunity},
 we find that the last expression equals 
 \begin{align*}
  &
  \sum_{q \in [\rho^c]} \lambda_{q}
  \cartanlambda_{\rho} 
  +
  \sum_{p \in [\rho] } \cartanlambda_p
  \wedge \phi_\rho
  =
  \sum_{q \in [\rho^c]} \lambda_{q}
  \cartanlambda_{\rho} 
  +
  \sum_{p \in [\rho] } \cartanlambda_p
  \wedge \eps(p,\rho-p) \lambda_p \cartanlambda_{\rho-p}
  \\&\qquad
  =
  \sum_{q \in [\rho^c]} \lambda_{q}
  \cartanlambda_{\rho} 
  +
  \sum_{p \in [\rho]} \lambda_{p}
  \cartanlambda_{\rho} 
  =
  \sum_{i=0}^n \lambda_i \cartanlambda_\rho
  =
  \cartanlambda_\rho
  ,
 \end{align*}
 which had to be shown. 
\springerqed\end{proof}

The following identity describes an important linear dependence between Whitney forms of higher order;
see also \cite[Equation (6.5)]{AFWgeodecomp} and \cite[Proposition 3.3]{christiansen2013high}.

\begin{lemma} \label{prop:whitneycancellationlemma}
 Let $k \in [0:n]$ and $\rho \in \Sigma_{0}(k,n)$. Then 
 \begin{align}
  \label{math:whitneycancellationlemma}
  \sum_{p \in [\rho]} \eps(p,\rho-p) \lambda_{p} \phi_{\rho-p}
=
  0.
 \end{align}
\end{lemma}
\begin{proof}
 Using the definition of Whitney forms \eqref{math:definitionwhitneyform}, we derive
\begin{align*}
\sum_{p \in [\rho]} \eps(p,\rho-p) \lambda_{p} \phi_{\rho-p}
  &=
  \sum_{p \in [\rho]} \eps(p,\rho-p) \lambda_{p}
  \sum_{s \in [\rho-p]} \lambda_{s} \eps(s,\rho-p-s) \cartanlambda_{\rho-p-s}
  \\&=
  \sum_{ \substack{ p,s \in [\rho] \\ p \neq s } }
  \eps(p,\rho-p) \eps(s,\rho-p-s) \lambda_{p} \lambda_{s}
  \cartanlambda_{\rho-p-s}
  .
 \end{align*}
 We have 
$\eps(s,\rho-p-s) = \eps(s,\rho-s) \eps( s, p )$
for any $s,p \in [\rho]$ with $s \neq p$. 
It is now evident that the summands cancel out.
\springerqed\end{proof}

\section{Finite Element Spaces}
\label{sec:finiteelementspaces}

In this section we discuss two families of barycentric differential forms over simplices
and find geometrically decomposed bases. 
Throughout this section, we let $T$ be a simplex of dimension $n$,
let $r \in \bbN_{0}$, and let $k \in [0:n]$.

We are particularly interested in the following two spaces
of barycentric differential forms: 
\begin{align}
 \label{math:finiteelementspace:initial}
 \calP_r\Lambda^{k}(T)
 &:=
 \linhull 
 \left\{\;
  \lambda_T^{\alpha} \cartanlambda^{T}_{\sigma} 
  \suchthat* 
  \alpha \in A(r,n), \; \sigma \in \Sigma(k,n)
 \;\right\}
 ,
 \\
 \label{math:finiteelementspace:whitney}
 \calP^{-}_r\Lambda^{k}(T)
 &:=
 \linhull 
 \left\{\;
  \lambda_{T}^{\alpha} \phi^{T}_{\rho}
  \suchthat* 
  \alpha \in A(r-1,n), \; \rho \in \Sigma_{0}(k,n)
 \;\right\}
 .
\end{align}
We also consider subspaces of differential forms with vanishing traces: 
\begin{align}
 \mathring\calP_r\Lambda^{k}(T)
 &:= 
 \left\{\; 
  \omega \in \calP_r\Lambda^{k}(T)
  \suchthat*  
  \forall F \subsetneq T : \trace_{T,F} \omega = 0 
 \;\right\}
 ,
 \\
 \mathring\calP_r^{-}\Lambda^{k}(T)
 &:= 
 \left\{\; 
  \omega \in \calP_r^{-}\Lambda^{k}(T)
  \suchthat*  
  \forall F \subsetneq T : \trace_{T,F} \omega = 0 
 \;\right\}
 .
\end{align}
It is evident that these spaces are nested, 
as follows from Lemma~\ref{prop:differentialdecomposition} and definitions,
\begin{gather*}
 \calP_r\Lambda^{k}(T) \subseteq \calP_{r+1}^{-}\Lambda^{k}(T) \subseteq \calP_{r+1}\Lambda^{k}(T)
 ,
 \\
 \mathring\calP_r\Lambda^{k}(T) \subseteq \mathring\calP_{r+1}^{-}\Lambda^{k}(T) \subseteq \mathring\calP_{r+1}\Lambda^{k}(T)
 ,
\end{gather*}
and, by the results in Subsection~\ref{subsec:traces}, 
that they are closed under taking traces: 
if $F \subseteq T$ is a subsimplex, then
\begin{gather*}
 \trace_{T,F} \calP_r\Lambda^{k}(T) = \calP_r\Lambda^{k}(F), 
 \quad 
 \trace_{T,F} \calP_r^{-}\Lambda^{k}(T) = \calP_r^{-}\Lambda^{k}(F).
\end{gather*}
We remark that our definitions \eqref{math:finiteelementspace:initial} and \eqref{math:finiteelementspace:whitney}
are equivalent to the respective definitions in \cite{AFW1}.

\subsection{Basis construction for $\calP_r\Lambda^{k}(T)$ and $\mathring\calP_r\Lambda^{k}(T)$}
\label{subsec:higherorderforms}

In this subsection, we study spanning sets and bases 
for the spaces $\calP_{r}\Lambda^{k}(T)$ and $\mathring\calP_{r}\Lambda^{k}(T)$.
We introduce the sets of barycentric differential forms 
\begin{align}
 \label{math:polycanonicalspanningset}
 \calS\calP_r\Lambda^{k}(T)
 &:=
 \left\{\;
  \lambda_T^{\alpha} \cartanlambda^{T}_{\sigma} 
  \suchthat* 
  \alpha \in A(r,n), 
  \;
  \sigma \in \Sigma(k,n)
 \;\right\},
 \\
 \label{math:polymathringspanningset}
 \calS\mathring\calP_r\Lambda^{k}(T)
 &:=
 \left\{\;
  \lambda_T^{\alpha}\cartanlambda^T_{\sigma}
  \suchthat* 
  \begin{array}{l}
   \alpha \in A(r,n), \; \sigma \in \Sigma(k,n),
   \\{}
   [\alpha]\cup[\sigma] = [0:n]
  \end{array}
 \;\right\}.
\end{align}
Furthermore, under the restriction that $r \geq 1$, 
we consider the sets of barycentric differential forms 
\begin{align}
 \label{math:polybetterbasis}
 \calB\calP_r\Lambda^{k}(T) 
 &:=
 \left\{\;
  \lambda_T^{\alpha} \cartanlambda^T_{\sigma}
  \suchthat* 
  \alpha \in A(r,n), \; \sigma \in \Sigma(k,n), \; \lfloor\alpha\rfloor \notin [\sigma]
 \;\right\},
 \\
 \label{math:polymathringbasis}
 \calB\mathring\calP_r\Lambda^{k}(T)
 &:=
 \left\{\;
  \lambda_T^{\alpha}\cartanlambda^T_{\sigma}
  \suchthat* 
  \begin{array}{l}
   \alpha \in A(r,n), \; \sigma \in \Sigma(k,n), \; \lfloor\alpha\rfloor \notin [\sigma], 
   \\
   {}[\alpha]\cup[\sigma] = [0:n]
  \end{array}
 \;\right\}.
\end{align}
We call $\calS\calP_{r}\Lambda^{k}(T)$ the \emph{canonical spanning set} of $\calP_{r}\Lambda^{k}(T)$,
and we call $\calS\mathring\calP_{r}\Lambda^{k}(T)$ the \emph{canonical spanning set} of $\mathring\calP_{r}\Lambda^{k}(T)$;
these names are justified below. 
Evidently, 
\begin{gather*}
 \calB\mathring\calP_r\Lambda^{k}(T)
 \subseteq  
 \calS\mathring\calP_r\Lambda^{k}(T),
 \quad 
 \calS\mathring\calP_r\Lambda^{k}(T)
 \subseteq  
 \calS\calP_r\Lambda^{k}(T),
 \\ 
 \calB\mathring\calP_r\Lambda^{k}(T)
 \subseteq  
 \calB\calP_r\Lambda^{k}(T),
 \quad 
 \calB\calP_r\Lambda^{k}(T)
 \subseteq  
 \calS\calP_r\Lambda^{k}(T). 
\end{gather*}
Suppose that $F \subseteq T$ is a subsimplex. 
It is clear from definitions that 
\begin{gather*}
 \trace_{T,F} \calS\calP_r\Lambda^{k}(T) = \calS\calP_r\Lambda^{k}(F), 
 \quad 
 \trace_{T,F} \calB\calP_r\Lambda^{k}(T) = \calB\calP_r\Lambda^{k}(F).
\end{gather*}
In fact, the trace of any member of $\calS\calP_r\Lambda^{k}(T)$ onto $F$ is either zero
or a member of $\calS\calP_r\Lambda^{k}(F)$, and any member of $\calS\calP_r\Lambda^{k}(F)$
has exactly one preimage in $\calS\calP_r\Lambda^{k}(T)$ under the trace.
More specifically, 
if $\lambda_T^{\alpha}\cartanlambda^T_{\sigma} \in \calS\calP_r\Lambda^{k}(T)$ 
with $[\alpha]\cup[\sigma] \subseteq [\imath(F,T)]$, then
\begin{align*}
 \trace_{T,F} \lambda_T^{\alpha} \cartanlambda^T_{\sigma}
 =
 \lambda_F^{ \widehat\alpha } \cartanlambda^F_{ \widehat\sigma }
 \in
 \calS\calP_r\Lambda^{k}(F),
\end{align*}
where $\widehat\alpha = \alpha \circ \imath(F,T)$
and $\widehat\sigma = \imath(F,T)^{\dagger} \circ \sigma$.
Note also that $\lfloor\alpha\rfloor \notin [\sigma]$ implies $\lfloor\widehat\alpha\rfloor \notin [\widehat\sigma]$.
In turn, if $\lambda^{\alpha}_F\cartanlambda^F_{\sigma} \in \calS\calP_r\Lambda^{k}(F)$,
then
\begin{align*}
 \lambda_T^{ \widetilde\alpha } \cartanlambda^T_{ \widetilde\sigma }
 \in
 \calS\calP_r\Lambda^{k}(T),
 \quad
 \trace_{T,F} 
 \lambda_T^{ \widetilde\alpha } \cartanlambda^T_{ \widetilde\sigma }
 =
 \lambda_F^{\alpha}\cartanlambda^F_{\sigma},
\end{align*}
where $\widetilde\alpha = \alpha \circ \imath(F,T)^{\dagger}$ over $[\imath(F,T)]$ and zero otherwise, 
and where $\widetilde\sigma = \imath(F,T) \circ \sigma$.
Note also that $\lfloor\alpha\rfloor \notin [\sigma]$ implies $\lfloor\widetilde\alpha\rfloor \notin [\widetilde\sigma]$.
\\

Note that $\calP_r\Lambda^{k}(T) = \linhull \calS\calP_r\Lambda^{k}(T)$ by definition,
which is the reason we call $\calS\calP_{r}\Lambda^{k}(T)$ the canonical spanning set of $\calP_r\Lambda^{k}(T)$. 
Whereas $\calS\calP_{r}\Lambda^{k}(T)$ is generally not linearly independent and hence is not a basis,
its subset $\calB\calP_r\Lambda^{k}(T)$ is a basis, as we now show.

\begin{theorem}
 Let $r \geq 1$. The set $\calB\calP_r\Lambda^{k}(T)$ is a basis of $\calP_{r}\Lambda^{k}(T)$.
\end{theorem}
\begin{proof}
 The claim holds in the case $k=0$, so let us assume that $k > 0$. 
 First we show that $\calB\calP_r\Lambda^{k}(T)$ spans $\calP_{r}\Lambda^{k}(T)$.
 For any $\alpha \in A(r,n)$ and $\sigma \in \Sigma(k,n)$ with $\lfloor\alpha\rfloor \in [\sigma]$ we find
 \begin{align*}
  \lambda^{\alpha}_{T} \cartanlambda^T_{\sigma}
  &=
  \eps({\lfloor\alpha\rfloor},\sigma-{\lfloor\alpha\rfloor})
  \lambda^{\alpha}_{T} \cartanlambda^T_{\lfloor\alpha\rfloor} \wedge \cartanlambda^T_{\sigma-{\lfloor\alpha\rfloor}}
  \\&=
  -
  \eps({\lfloor\alpha\rfloor},\sigma-{\lfloor\alpha\rfloor})
  \sum_{q \in [\sigma^c]} 
  \lambda^{\alpha}_{T} \cartanlambda^T_q \wedge \cartanlambda^T_{\sigma-{\lfloor\alpha\rfloor}}
  \\&=
  - 
  \eps({\lfloor\alpha\rfloor},\sigma-{\lfloor\alpha\rfloor})
  \sum_{q \in [\sigma^c]} 
  \eps(q,\sigma-{\lfloor\alpha\rfloor})
  \lambda^{\alpha}_{T} 
  \cartanlambda^T_{\sigma-{\lfloor\alpha\rfloor}+q}
.
\end{align*}
 The latter is a linear combination of members of $\calB\calP_r\Lambda^{k}(T)$.
 Hence $\calB\calP_r\Lambda^{k}(T)$ spans $\calP_{r}\Lambda^{k}(T)$. 
 It remains to show that $\calB\calP_r\Lambda^{k}(T)$ is linearly independent.
 Let $\omega \in \calP_{r}\Lambda^{k}(T)$.
Then there exist coefficients $\omega_{\alpha\sigma} \in \bbC$ such that
 \begin{align*}
  \omega
  =
  \sum_{ \substack{ \alpha \in A(r,n) } }
  \sum_{ \substack{ \sigma \in \Sigma(k,n) \\ \lfloor\alpha\rfloor \notin [\sigma] } }
  \omega_{\alpha\sigma} \lambda_T^{\alpha} \cartanlambda^T_{\sigma}
  .
 \end{align*}
 Suppose $\omega = 0$ while not all coefficients $\omega_{\alpha\sigma}$ vanish.
 Consider the constant $k$-forms
 \begin{align*}
  V_{\alpha}
  :=
  \sum_{ \substack{ \sigma \in \Sigma(k,n) \\ \lfloor\alpha\rfloor \notin [\sigma] } }
  \omega_{\alpha\sigma} \cartanlambda^T_{\sigma},
  \quad
  \alpha \in A(r,n).
 \end{align*}
 For every $\alpha \in A(r,n)$, we have $V_{\alpha} = 0$
 if and only if 
 for all $\sigma \in \Sigma(k,n)$ with $\lfloor\alpha\rfloor \notin [\sigma]$ we have $\omega_{\alpha\sigma} = 0$;
 that holds because those $\cartanlambda^T_{\sigma}$ with $\lfloor\alpha\rfloor \notin [\sigma]$ are linearly independent.
 Since we suppose that not all coefficients vanish, 
 there exists $\alpha \in A(r,n)$ with $V_{\alpha} \neq 0$.
 Letting $\calV_{\alpha}$ be the constant $k$-vector field dual to $V_{\alpha}$, 
\begin{align*}
  0
  &=
  \omega(\calV_{\alpha})
  =
  \sum_{\beta \in A(r,n)} \lambda_T^{\beta} V_{\beta}(\calV_{\alpha})
  =
  \lambda_T^{\alpha}
  +
  \sum_{ \substack{ \beta \in A(r,n) \\ \beta \neq \alpha } } 
  \lambda_T^{\beta} V_{\beta}(\calV_{\alpha}).
\end{align*}
 But this contradicts the linear independence of the set $\calB\calP_{r}\Lambda^{0}(T)$. 
 Hence all coefficients must vanish.
 This shows linear independence, completing the proof. 
\springerqed\end{proof}

Next we show 
that the subset $\calS\mathring\calP_r\Lambda^{k}(T) \subseteq \calS\calP_r\Lambda^{k}(T)$
spans the subspace $\mathring\calP_r\Lambda^{k}(T) \subseteq \calP_r\Lambda^{k}(T)$
and 
that the subset $\calB\mathring\calP_r\Lambda^{k}(T) \subseteq \calB\calP_r\Lambda^{k}(T)$
is a basis for the same space.

\begin{theorem}
 Let $r \geq 1$. 
 The set $\calB\mathring\calP_r\Lambda^{k}(T)$ is a basis for $\mathring\calP_{r}\Lambda^{k}(T)$,
 and $\calS\mathring\calP_r\Lambda^{k}(T)$ is a spanning set for that space.
\end{theorem}
\begin{proof}
 Let $\omega \in \mathring\calP_r\Lambda^{k}(T)$.
 Then $\omega \in \calP_r\Lambda^{k}(T)$,
 and thus there exist unique coefficients $\omega_{\alpha\sigma} \in \bbC$ such that
 \begin{align*}
  \omega
  =
  \sum_{ \substack{ \alpha \in A(r,n) } }
  \sum_{ \substack{ \sigma \in \Sigma(k,n) \\ \lfloor\alpha\rfloor \notin [\sigma] } }
  \omega_{\alpha\sigma} \lambda_T^{\alpha}\cartanlambda^T_{\sigma}.
 \end{align*}
 Suppose that $F$ is a lower-dimensional subsimplex of $T$.
 Since $\omega \in \mathring\calP_{r}\Lambda^{k}(T)$, we have $0 = \trace_{T,F} \omega$. 
 So 
\begin{align*}
  0
  &=
\sum_{ \substack{ \alpha \in A(r,n) \\ \sigma \in \Sigma(k,n) \\ \lfloor\alpha\rfloor \notin [\sigma] } }
  \omega_{\alpha\sigma} \trace_{T,F} \lambda_T^{\alpha}\cartanlambda^T_{\sigma}
= 
\sum_{ \substack{ \alpha \in A(r,n) \\ \sigma \in \Sigma(k,n) \\ \lfloor\alpha\rfloor \notin [\sigma] \\ [\alpha]\cup[\sigma]\subseteq[\imath(F,T)] } }
  \omega_{\alpha\sigma} \lambda^{\alpha \circ \imath(F,T)}_F \cartanlambda^F_{\imath(F,T)^{\dagger} \circ \sigma}
  .
 \end{align*}
 Since the last sum is written in terms of the basis $\calB\calP_{r}\Lambda^{k}(F)$,
 we must have $\omega_{\alpha\sigma} = 0$ for all $[\alpha]\cup[\sigma] \subseteq [\imath(F,T)]$. 
 Since we assumed $F$ to be an arbitrary proper subsimplex of $T$,
 we get that $\omega_{\alpha\sigma} = 0$
 when $[\alpha]\cup[\sigma] \neq [0:n]$. 
We conclude that $\calB\mathring\calP_r\Lambda^{k}(T)$ is a spanning set
 of $\mathring\calP_r\Lambda^{k}(T)$.
 It is linearly independent too, being a subset of $\calB\calP_r\Lambda^{k}(T)$.
 Hence $\calB\mathring\calP_r\Lambda^{k}(T)$ is a basis, as claimed,
 and $\calS\mathring\calP_r\Lambda^{k}(T)$ being a spanning set is a trivial consequence. 
\springerqed\end{proof}

\begin{remark}
 We compare our construction with the previous results in the literature. 
 Our basis $\calB\mathring\calP_{r}\Lambda^{k}(T)$ has been displayed explicitly in previous works, 
 albeit in a different form. 
 To illustrate that, we note that $\calB\mathring\calP_{r}\Lambda^{k}(T)$ can also be written as 
 \begin{align*}
  \calB\mathring\calP_r\Lambda^{k}(T)
  &
  =
  \left\{\;
   \lambda_T^{\alpha}\cartanlambda^T_{\sigma}
   \suchthat* 
   \begin{array}{l}
    \alpha \in A(r,n), \; \sigma \in \Sigma(k,n), 
    \\
    \lfloor\alpha\rfloor \notin [\sigma], \; [\alpha]\cup[\sigma] = [0:n]
   \end{array}
  \;\right\}
  \\&
  =
  \left\{\;
   \lambda_T^{\alpha}\cartanlambda^T_{\sigma}
   \suchthat* 
   \begin{array}{l}
    \alpha \in A(r,n), \; \sigma \in \Sigma(k,n), 
    \\
    \lfloor\alpha\rfloor \in [0:n]\setminus [\sigma], \; [\alpha]\cup[\sigma] = [0:n]
   \end{array}
  \;\right\}
  \\&
  =
  \left\{\;
   \lambda_T^{\alpha}\cartanlambda^T_{\sigma}
   \suchthat* 
   \begin{array}{l}
    \alpha \in A(r,n), \; \sigma \in \Sigma(k,n), 
    \\
    \lfloor\alpha\rfloor = \min([0:n]\setminus [\sigma]), \; [\alpha]\cup[\sigma] = [0:n]
   \end{array}
  \;\right\}
  \\&
  =
  \left\{\;
   \lambda_T^{\alpha}\cartanlambda^T_{\sigma}
   \suchthat* 
   \begin{array}{l}
    \alpha \in A(r,n), \; \sigma \in \Sigma(k,n), 
    \\
    \lfloor\alpha\rfloor \geq \min([0:n]\setminus [\sigma]), \; [\alpha]\cup[\sigma] = [0:n]
   \end{array}
  \;\right\}.
 \end{align*}
Due to the identities above, we now see that $\calB\mathring\calP_{r}\Lambda^{k}(T)$
 agrees with the basis description in Theorem~6.1 of \cite{AFWgeodecomp}. 
 That same basis of $\mathring\calP\Lambda^{k}(T)$ is also used implicitly in Theorem~4.22 of \cite{AFW1}.
 However, our basis $\calB\calP_{r}\Lambda^{k}(T)$ of $\calP_{r}\Lambda^{k}(T)$ is not explicitly described there.  
\end{remark}

\begin{remark}
    \label{remark:vectoranalysis}
Our basis for $\calP_{r}\Lambda^{k}(T)$ has apparently not been mentioned in the literature before,
    so we give some examples in the language of vector analysis. 
Over a triangle, a basis for the Brezzi-Douglas-Marini space of polynomial degree $r$ is 
    \begin{align}
        \left\{\;
            \lambda_T^{\alpha} \nablalambda^T_{p}
            \suchthat* 
            \alpha \in A(r,2), \; p \in \{0,1,2\}, \; \lfloor\alpha\rfloor \neq p
        \;\right\}
        .
    \end{align}
    Over a tetrahedron, 
    we have a basis for the curl-conforming N\'ed\'elec elements 
    of the second kind of polynomial degree $r$
    \begin{align}
        \left\{\;
            \lambda_T^{\alpha} \nablalambda^T_{p}
            \suchthat* 
            \alpha \in A(r,3), \; p \in \{0,1,2,3\}, \; \lfloor\alpha\rfloor \neq p
        \;\right\}
        ,
    \end{align}
    and we have a basis for the divergence-conforming N\'ed\'elec elements 
    of the second kind of polynomial degree $r$
    \begin{align}
        \left\{\;
            \lambda_T^{\alpha} \nablalambda^T_{p} \times \nablalambda^T_{q}
            \suchthat* 
            \alpha \in A(r,3), \; p,q \in \{0,1,2,3\}, \; p < q, \; \lfloor\alpha\rfloor \notin \{p,q\}
        \;\right\}
        .
    \end{align}
    These bases are already geometrically decomposed: for each member of these bases, 
    the indices in the parameters $[\alpha]$, $p$, and $q$ completely determine which 
    subsimplex the corresponding member is associated with.
    For example, $\lambda_{0}\nablalambda_{1}^{T}$ and $\lambda_{1}\nablalambda_{0}^{T}$ are associated with the edge of $T$ 
    that contains the zeroth and the first vertices,
    and $\lambda_{0}\lambda_{1}\nablalambda_{2}^{T}\times\nablalambda_{3}^{T}$
    is associated with the entire tetrahedron $T$.
    Tables~\ref{table:tableofbasisforms_21}~--~\ref{table:tableofbasisforms_32}
    illustrate these bases for vector-valued finite elements for low polynomial degrees. 
\end{remark}

\begin{table}[t]
  \centering
  \begin{tabular}{r|l}
  \label{tableofbasisforms}
  $r=1$
  &
  $\lambda_0\{\nablalambda_1,\nablalambda_2\}$,
  $\lambda_1\{\nablalambda_0,\nablalambda_2\}$,
  $\lambda_2\{\nablalambda_0,\nablalambda_1\}$
  \\[0.25cm]
  $r=2$
  &
  $\lambda_0 \{\lambda_0,\lambda_1,\lambda_2\} \{\nablalambda_1,\nablalambda_2\}$,
  $\lambda_1 \{\lambda_1,\lambda_2\} \{\nablalambda_0,\nablalambda_2\}$,
  $\lambda_2^{2} \{\nablalambda_0,\nablalambda_1\}$
  \\[0.25cm]
  $r=3$
  &
  $\lambda_0 \{\lambda_0,\lambda_1,\lambda_2\}^{2} \{\nablalambda_1,\nablalambda_2\}$,
  $\lambda_1 \{\lambda_1,\lambda_2\}^{2} \{\nablalambda_0,\nablalambda_2\}$,
  $\lambda_2^{3} \{\nablalambda_0,\nablalambda_1\}$
\end{tabular}
  \caption{Bases for the Brezzi-Douglas-Marini space on a $2$-simplex $T$ for low polynomial degree $r$ in terms of barycentric coordinates.}
  \label{table:tableofbasisforms_21}
\end{table}

\begin{table}[t]
\centering
  \begin{tabular}{r|l}
  $r=1$
  &
  $\lambda_0\{\nablalambda_1,\nablalambda_2,\nablalambda_3\}$,
  $\lambda_1\{\nablalambda_0,\nablalambda_2,\nablalambda_3\}$,
  $\lambda_2\{\nablalambda_0,\nablalambda_1,\nablalambda_3\}$,
  \\ & 
  $\lambda_3\{\nablalambda_0,\nablalambda_1,\nablalambda_2\}$
  \\[0.25cm]
  $r=2$
  &
  $\lambda_0 \{\lambda_0,\lambda_1,\lambda_2,\lambda_3\} \{\nablalambda_1,\nablalambda_2,\nablalambda_3\}$,
  $\lambda_1 \{\lambda_1,\lambda_2,\lambda_3\} \{\nablalambda_0,\nablalambda_2,\nablalambda_3\}$,
  \\ & 
  $\lambda_2 \{\lambda_2,\lambda_3\} \{\nablalambda_0,\nablalambda_1,\nablalambda_3\}$,
  $\lambda_3^{2} \{\nablalambda_0,\nablalambda_1,\nablalambda_2\}$,
  \\[0.25cm]
  $r=3$
  &
  $\lambda_0 \{\lambda_0,\lambda_1,\lambda_2,\lambda_3\}^{2} \{\nablalambda_1,\nablalambda_2,\nablalambda_3\}$,
  $\lambda_1 \{\lambda_1,\lambda_2,\lambda_3\}^{2} \{\nablalambda_0,\nablalambda_2,\nablalambda_3\}$,
  \\ & 
  $\lambda_2 \{\lambda_2,\lambda_3\}^{2} \{\nablalambda_0,\nablalambda_1,\nablalambda_3\}$,
  $\lambda_3^{3} \{\nablalambda_0,\nablalambda_1,\nablalambda_2\}$,
\end{tabular}
  \caption{Bases for the curl-conforming N\'ed\'elec space of the second kind on a $3$-simplex $T$ for low polynomial degree $r$ in terms of barycentric coordinates.}
  \label{table:tableofbasisforms_31}
\end{table}

\begin{table}[t]
\centering
  \begin{tabular}{r|l}
  $r=1$
  &
  $\lambda_0
  \{\nablalambda_1\times\nablalambda_2, \nablalambda_1\times\nablalambda_3, \nablalambda_2\times\nablalambda_3\}$,
  \\ & 
  $\lambda_1
  \{\nablalambda_0\times\nablalambda_2, \nablalambda_0\times\nablalambda_3, \nablalambda_2\times\nablalambda_3\}$,
  \\ & 
  $\lambda_2
  \{\nablalambda_0\times\nablalambda_1, \nablalambda_0\times\nablalambda_3, \nablalambda_1\times\nablalambda_3\}$,
  \\ & 
  $\lambda_3
  \{\nablalambda_0\times\nablalambda_1, \nablalambda_0\times\nablalambda_2, \nablalambda_1\times\nablalambda_2\}$
  \\[0.25cm]
  $r=2$
  &
  $\lambda_0 \{\lambda_0,\lambda_1,\lambda_2,\lambda_3\}
  \{\nablalambda_1\times\nablalambda_2, \nablalambda_1\times\nablalambda_3, \nablalambda_2\times\nablalambda_3\}$,
  \\ & 
  $\lambda_1 \{\lambda_1,\lambda_2,\lambda_3\}
  \{\nablalambda_0\times\nablalambda_2, \nablalambda_0\times\nablalambda_3, \nablalambda_2\times\nablalambda_3\}$,
  \\ & 
  $\lambda_2 \{\lambda_2,\lambda_3\}
  \{\nablalambda_0\times\nablalambda_1, \nablalambda_0\times\nablalambda_3, \nablalambda_1\times\nablalambda_3\}$,
  \\ & 
  $\lambda_3^{2}
  \{\nablalambda_0\times\nablalambda_1, \nablalambda_0\times\nablalambda_2, \nablalambda_1\times\nablalambda_2\}$,
  \\[0.25cm]
  $r=3$
  &
  $\lambda_0 \{\lambda_0,\lambda_1,\lambda_2,\lambda_3\}^{2}
  \{\nablalambda_1\times\nablalambda_2, \nablalambda_1\times\nablalambda_3, \nablalambda_2\times\nablalambda_3\}$,
  \\ & 
  $\lambda_1 \{\lambda_1,\lambda_2,\lambda_3\}^{2}
  \{\nablalambda_0\times\nablalambda_2, \nablalambda_0\times\nablalambda_3, \nablalambda_2\times\nablalambda_3\}$,
  \\ & 
  $\lambda_2 \{\lambda_2,\lambda_3\}^{2}
  \{\nablalambda_0\times\nablalambda_1, \nablalambda_0\times\nablalambda_3, \nablalambda_1\times\nablalambda_3\}$,
  \\ & 
  $\lambda_3^{3} 
  \{\nablalambda_0\times\nablalambda_1, \nablalambda_0\times\nablalambda_2, \nablalambda_1\times\nablalambda_2\}$,
\end{tabular}
  \caption{Bases for the divergence-conforming N\'ed\'elec space of the second kind on a $3$-simplex $T$ for low polynomial degree $r$ in terms of barycentric coordinates.}
  \label{table:tableofbasisforms_32}
\end{table}

\subsection{Basis construction for $\calP^-_r\Lambda^{k}(T)$ and $\mathring\calP^-_r\Lambda^{k}(T)$}
\label{subsec:higherorderwhitneyforms}

This subsection follows a similar path as the previous one.  
We study spanning sets and bases 
for the spaces $\calP_{r}^{-}\Lambda^{k}(T)$ and $\mathring\calP_{r}^{-}\Lambda^{k}(T)$.
Under the restriction that $r \geq 1$, 
we introduce the sets of barycentric differential forms 
\begin{align}
 \label{math:whitneyspanningset}
 \calS\calP^{-}_r\Lambda^{k}(T)
 &:=
 \left\{\;
  \lambda_{T}^{\alpha} \phi^{T}_{\rho}
  \suchthat* 
  \alpha \in A(r-1,n), \; \rho \in \Sigma_{0}(k,n)
 \;\right\},
 \\
 \label{math:whitneymathringspanningset}
 \calS\mathring\calP^{-}_r\Lambda^{k}(T)
 &:=
 \left\{\; 
  \lambda_T^{\alpha}\phi_{\rho}^{T}
  \suchthat* 
  \begin{array}{l}
   \alpha \in A(r-1,n), \; \rho \in \Sigma_{0}(k,n),
   \\{}
   [\alpha]\cup[\rho]=[0:n]
  \end{array}
 \;\right\},
\end{align}
and their subsets 
\begin{align}
 \label{math:whitneybasis}
 \calB\calP^{-}_r\Lambda^{k}(T)
 &:=
 \left\{\; 
  \lambda_T^{\alpha}\phi_{\rho}^{T}
  \suchthat* 
   \begin{array}{l}
    \alpha \in A(r-1,n), \; \rho \in \Sigma_{0}(k,n),
    \\
    \lfloor\alpha\rfloor \geq \lfloor\rho\rfloor 
   \end{array} 
 \;\right\},
 \\
 \label{math:whitneymathringbasis}
 \calB\mathring\calP^{-}_r\Lambda^{k}(T)
 &:=
 \left\{\;
  \lambda_T^{\alpha}\phi_{\rho}^{T}
  \suchthat* 
   \begin{array}{l}
    \alpha \in A(r-1,n), \; \rho \in \Sigma_{0}(k,n),
    \\
    \lfloor\alpha\rfloor \geq \lfloor\rho\rfloor, \; [\alpha]\cup[\rho]=[0:n]
   \end{array} 
 \;\right\}.
\end{align}
We call $\calS\calP_{r}^{-}\Lambda^{k}(T)$ the \emph{canonical spanning set} of $\calP_{r}^{-}\Lambda^{k}(T)$,
and we call $\calS\mathring\calP_{r}^{-}\Lambda^{k}(T)$ the \emph{canonical spanning set} of $\mathring\calP_{r}^{-}\Lambda^{k}(T)$;
as before, these names will be justified shortly. 
It is evident that \begin{gather*}
 \calB\mathring\calP_r^{-}\Lambda^{k}(T)
 \subseteq  
 \calS\mathring\calP_r^{-}\Lambda^{k}(T),
 \quad 
 \calS\mathring\calP_r^{-}\Lambda^{k}(T)
 \subseteq  
 \calS\calP_r^{-}\Lambda^{k}(T),
 \\ 
 \calB\mathring\calP_r^{-}\Lambda^{k}(T)
 \subseteq  
 \calB\calP_r^{-}\Lambda^{k}(T),
 \quad 
 \calB\calP_r^{-}\Lambda^{k}(T)
 \subseteq  
 \calS\calP_r^{-}\Lambda^{k}(T). 
\end{gather*}
Suppose that $F \subseteq T$ is a subsimplex. From definitions it is clear that 
\begin{gather*}
 \trace_{T,F} \calS\calP_r^{-}\Lambda^{k}(T) = \calS\calP_r^{-}\Lambda^{k}(F), 
 \quad 
 \trace_{T,F} \calB\calP_r^{-}\Lambda^{k}(T) = \calB\calP_r^{-}\Lambda^{k}(F).
\end{gather*}
In fact, the trace of any member of $\calS\calP_r^{-}\Lambda^{k}(T)$ onto $F$ is either zero
or a member of $\calS\calP_r^{-}\Lambda^{k}(F)$, and any member of $\calS\calP_r^{-}\Lambda^{k}(F)$
has exactly one preimage in $\calS\calP_r^{-}\Lambda^{k}(T)$ under the trace.
Again, we make this more specific. 
Whenever $\lambda_T^{\alpha}\phi^T_{\rho} \in \calS\calP^{-}_r\Lambda^{k}(T)$ with $[\alpha]\cup[\rho] \subseteq [\imath(F,T)]$,
then
\begin{align*}
 \trace_{T,F} \lambda_T^{\alpha}\phi^T_{\rho}
 =
 \lambda_F^{ \widehat\alpha } \phi^F_{ \widehat\rho }
 \in
 \calS\calP^{-}_r\Lambda^{k}(F),
\end{align*}
where $\widehat\alpha = \alpha \circ \imath(F,T)$
and $\widehat\rho = \imath(F,T)^{\dagger} \circ \rho$.
Note also that $\lfloor\alpha\rfloor \geq \lfloor\rho\rfloor$ implies $\lfloor\widehat\alpha\rfloor \geq \lfloor\widehat\rho\rfloor$,
by our assumption on the orders of vertices.
In turn, if $\lambda_F^{\alpha}\phi^F_{\rho} \in \calS\calP^{-}_r\Lambda^{k}(F)$,
then
\begin{align*}
 \lambda_T^{ \widetilde\alpha } \phi^T_{ \widetilde\rho }
 \in
 \calS\calP^{-}_r\Lambda^{k}(T),
 \quad
 \trace_{T,F}
 \lambda_T^{ \widetilde\alpha } \phi^T_{ \widetilde\rho }
 =
 \lambda^{\alpha}_F \phi^F_{\rho},
\end{align*}
where $\widetilde\alpha = \alpha \circ \imath(F,T)^{\dagger}$ over $[\imath(F,T)]$ and zero otherwise, 
and where $\widetilde\rho= \imath(F,T) \circ \rho$.
Note also that $\lfloor\alpha\rfloor \geq \lfloor\rho\rfloor$ implies $\lfloor\widetilde\alpha\rfloor \geq \lfloor\widetilde\rho\rfloor$.
\\

Similar as above, $\calP_r^{-}\Lambda^{k}(T) = \linhull \calS\calP_r^{-}\Lambda^{k}(T)$ by definitions,
so $\calS\calP_{r}\Lambda^{k}(T)$ is aptly called the canonical spanning set of the space $\calP_r\Lambda^{k}(T)$
of higher order Whitney forms. 
But $\calS\calP_{r}^{-}\Lambda^{k}(T)$ is generally not a basis for that space. 
Analogously to the previous subsection, 
we show 
that its subset $\calB\calP_r^{-}\Lambda^{k}(T)$ is a basis,
that $\calS\mathring\calP^{-}_r\Lambda^{k}(T)$ is a spanning set, 
and that $\calB\mathring\calP^{-}_r\Lambda^{k}(T)$ is a basis of $\mathring\calP_r^{-}\Lambda^{k}(T)$.

\begin{theorem}
  Let $r \geq 1$. The set $\calB\calP^{-}_r\Lambda^{k}(T)$ is a basis of $\calP^{-}_{r}\Lambda^{k}(T)$.
\end{theorem}
\begin{proof}
We first show that $\calB\calP^{-}_r\Lambda^{k}(T)$
  is a spanning set of $\calP_r^{-}\Lambda^{k}(T)$.
  In the case $r=1$, we have $\calB\calP^{-}_r\Lambda^{k}(T) = \calS\calP^{-}_r\Lambda^{k}(T)$ by definitions,
  so it remains to consider the case $r \geq 2$. 
  Let $\alpha \in A(r-1,n)$ and $\rho \in \Sigma_{0}(k,n)$,
  and write $p := \lfloor\alpha\rfloor$.
  There exists $\beta \in A(r-2,n)$ with $\lambda_T^\alpha = \lambda_T^\beta \lambda^T_{p}$.
  If $p \geq \lfloor\rho\rfloor$, then $\lambda_T^{\alpha} \phi^T_{\rho} \in \calB\calP^{-}_{r}\Lambda^{k}(T)$ by definition.
  If instead $p < \lfloor\rho\rfloor$,
  then Lemma~\ref{prop:whitneycancellationlemma} shows 
  \begin{align*}
  \lambda_T^{\alpha} \phi^T_{\rho}
  =
  \lambda_T^{\beta} \lambda^T_{p} \phi^T_{\rho}
  &
  =
  \lambda_T^{\beta} \eps( p, (\rho+p) - p ) \lambda^T_{p} \phi^T_{ (\rho+p) - p }
\\&
  =
-
  \sum_{ q \in [\rho] } 
  \eps( q, (\rho+p) - q ) 
  \lambda_T^{\beta}
  \lambda^T_{q} \phi^T_{ (\rho+p) - q }
.
  \end{align*} 
  Notice that the last sum is a linear combination of terms in $\calB\calP^{-}_r\Lambda^{k}(T)$
  because $p = \lfloor\rho+p\rfloor \leq \lfloor\beta+q\rfloor$.
  Hence all members of $\calS\calP^{-}_r\Lambda^{k}(T)$ are linear combinations 
  of members of $\calB\calP^{-}_r\Lambda^{k}(T)$,
  that is, $\calB\calP^{-}_r\Lambda^{k}(T)$ spans $\calP_r^{-}\Lambda^{k}(T)$.

  It remains to show that $\calB\calP^{-}_r\Lambda^{k}(T)$ is linearly independent.
  We perform an induction over the dimension of $T$.
  In the base case $n=k$, 
  the set $\calB\calP^{-}_r\Lambda^{k}(T)$ is linearly independent because it consists precisely of the terms $\lambda^\alpha \phi_T$ with $\alpha \in A(r-1,n)$. 
For the induction step, assume the statement is true for simplices of dimension up to $n-1$.
  Let $\omega \in \calP_{r}^{-}\Lambda^{k}(T)$ be written as 
  \begin{align*}
    \omega 
    =
    \sum_{ \substack{ \alpha \in A(r-1,n) } } 
    \sum_{ \substack{ \rho \in \Sigma_{0}(k,n) \\ \lfloor\alpha\rfloor \geq \lfloor\rho\rfloor } } 
    \omega_{\alpha\rho} \lambda_T^{\alpha} \phi^T_{\rho}
    ,
    \quad 
    \omega_{\alpha\rho} \in \bbC
    .
  \end{align*}
We need to show that all coefficients in that expansion vanish if $\omega$ vanishes.
  Assume that $\omega = 0$. Let $F$ be the subsimplex of $T$ of dimension $n-1$ not containing 
  the zeroth vertex of $T$.
  Then 
  \begin{align*}
    \trace_{T,F} \omega 
    = 
    \sum_{ \substack{ \alpha \in A(r-1,n) \\
    \rho \in \Sigma_{0}(k,n) \\
    \lfloor\alpha\rfloor \geq \lfloor\rho\rfloor } }
    \omega_{\alpha\rho} \trace_{T,F} \lambda_T^{\alpha}\phi^T_{\rho}
    =
    \sum_{ \substack{ \alpha \in A(r-1,n) \\
    \rho \in \Sigma_{0}(k,n) \\
    \lfloor\alpha\rfloor \geq \lfloor\rho\rfloor \\
    [\alpha]\cup[\rho] \subseteq [\imath(F,T)] } }
    \omega_{\alpha\rho} \lambda^{ \alpha \circ \imath(F,T) }_F \phi^F_{ \imath(F,T)^\dagger \circ \rho }
    .
   \end{align*}
   By the induction assumption, 
   this expresses $\trace_{T,F} \omega$ in terms of a basis of $\calP_r^{-}\Lambda^{k}(F)$.
   Recall that $0 = \trace_{T,F} \omega$.
   Hence $\omega_{\alpha\rho} = 0$ whenever $[\alpha]\cup[\rho] \subseteq [\imath(F,T)]$.
   Since $[\imath(F,T)] = [1:n]$ and $\lfloor\alpha\rfloor \geq \lfloor\rho\rfloor$,
   we see that $\omega_{\alpha\rho} = 0$ whenever $\lfloor\rho\rfloor \neq 0$.
   Thus
\begin{align*}
      \omega 
      =
      \sum_{ \substack{ \alpha \in A(r-1,n) } } 
      \sum_{ \substack{ \rho \in \Sigma_{0}(k,n) \\ \lfloor\rho\rfloor = 0 } } 
      \omega_{\alpha\rho} \lambda_T^{\alpha} \phi^T_{\rho}
      .
   \end{align*}
    Using the definition of Whitney forms, we write $\omega = \omega_{(0)} + \omega_{(+)}$, where 
    \begin{align*}
      \omega_{(0)} 
      &:=
      \sum_{ \substack{ \alpha \in A(r-1,n) } } 
      \sum_{ \substack{ \rho \in \Sigma_{0}(k,n) \\ \lfloor\rho\rfloor = 0 } } 
      \omega_{\alpha\rho} 
      \lambda_T^{\alpha} \lambda_{0}^{T}
      \cartanlambda^T_{\rho - 0}
      ,
      \\
      \omega_{(+)} 
      &:=
      \sum_{ \substack{ \alpha \in A(r-1,n) } } 
      \sum_{ \substack{ \rho \in \Sigma_{0}(k,n) \\ \lfloor\rho\rfloor = 0 } } 
      \sum_{ \substack{ p \in [\rho] \\ p > 0 } } 
      \omega_{\alpha\rho}
      \eps(p,\rho-p)
      \lambda_T^{\alpha} \lambda_{p}^{T} 
      \cartanlambda^T_{\rho - p}
    \end{align*}
  Notice that $\omega_{(0)}$ is expressed as a linear combination of terms in $\calB\calP_{r}\Lambda^{k}(T)$,
  which is a basis. 
  We use a recursive argument to prove that all $\omega_{\alpha\rho}$ vanish. 

 For the recursion step, let us assume that there exists $s \geq 0$ 
 such that $\omega_{\alpha\rho} = 0$ whenever $\alpha(0) > s$.
 Since the barycentric polynomials $\lambda_T^{\alpha}\lambda_{0}^{T}$ with $\alpha(0) = s$
 in the definition of $\omega_{(0)}$ always have a higher exponent in index $0$ 
 than any of the barycentric polynomials $\lambda_T^{\alpha}\lambda_{p}^{T}$ in the definition of $\omega_{(+)}$, 
 the identity $\omega_{(0)} + \omega_{(+)} = 0$ shows $\omega_{\alpha\rho} = 0$ also when $\alpha(0) = s$.
 We repeat this argument, beginning with the trivial choice $s = r-1$ and decreasing $s$ to zero, 
 and find that $\omega_{\alpha\rho} = 0$ for all coefficients. 
 Thus $\calB\mathring\calP^{-}_r\Lambda^{k}(T)$ is linearly independent.
 
 This closes the induction step, and the proof is complete.
\springerqed\end{proof}

\begin{theorem}
  Let $r \geq 1$. 
  The set $\calB\mathring\calP^{-}_r\Lambda^{k}(T)$ is a basis for $\mathring\calP^{-}_{r}\Lambda^{k}(T)$,
  and $\calS\mathring\calP^{-}_r\Lambda^{k}(T)$ is a spanning set for that space.
 \end{theorem}
 \begin{proof}
  Let $\omega \in \mathring\calP^{-}_r\Lambda^{k}(T)$.
  Then $\omega \in \calP^{-}_r\Lambda^{k}(T)$,
  and thus there exist unique coefficients $\omega_{\alpha\rho} \in \bbC$ such that
  \begin{align*}
   \omega 
   =
   \sum_{ \substack{ \alpha \in A(r-1,n) } } 
   \sum_{ \substack{ \rho \in \Sigma_{0}(k,n) \\ \lfloor\alpha\rfloor \geq \lfloor\rho\rfloor } } 
   \omega_{\alpha\rho} \lambda_T^{\alpha} \phi^T_{\rho}
   .
  \end{align*} Suppose that $F$ is a lower-dimensional subsimplex of $T$.
  Since $\omega \in \mathring\calP^{-}_{r}\Lambda^{k}(T)$, we have $0 = \trace_{T,F} \omega$. 
  So 
  \begin{align*}
   0
   &=
   \sum_{ \substack{ \alpha \in A(r-1,n) \\ \rho \in \Sigma_{0}(k,n) \\ \lfloor\alpha\rfloor \geq \lfloor\rho\rfloor } }
   \omega_{\alpha\rho} \trace_{T,F} \lambda_T^{\alpha} \phi^T_{\rho}
   = 
   \sum_{ \substack{ \alpha \in A(r-1,n) \\ \rho \in \Sigma_{0}(k,n) \\ \lfloor\alpha\rfloor \geq \lfloor\rho\rfloor \\ [\alpha]\cup[\rho]\subseteq[\imath(F,T)] } }
   \omega_{\alpha\rho} \lambda^{\alpha \circ \imath(F,T)}_F \phi^F_{ \imath(F,T)^\dagger \circ \rho }
   .
  \end{align*}
  Since the last sum is written in terms of the basis $\calB\calP^{-}_{r}\Lambda^{k}(F)$,
  we must have $\omega_{\alpha\rho} = 0$ for all $[\alpha]\cup[\rho] \subseteq [\imath(F,T)]$. 
  Since we assumed $F$ to be an arbitrary proper subsimplex of $T$,
  we get that $\omega_{\alpha\rho} = 0$
  when $[\alpha]\cup[\rho] \neq [0:n]$. 
  We conclude that $\calB\mathring\calP^{-}_r\Lambda^{k}(T)$ is a spanning set
  of $\mathring\calP^{-}_r\Lambda^{k}(T)$.
  It is linearly independent too, being a subset of $\calB\calP^{-}_r\Lambda^{k}(T)$.
  Hence $\calB\mathring\calP^{-}_r\Lambda^{k}(T)$ is a basis, as claimed,
  and $\calS\mathring\calP^{-}_r\Lambda^{k}(T)$ being a spanning set is a trivial consequence. 
 \springerqed\end{proof}

\begin{remark}
 \label{rem:basis_comparison_whitney}
 The bases for $\calP^{-}_{r}\Lambda^{k}(T)$ and $\mathring\calP^{-}_{r}\Lambda^{k}(T)$
 are identical to the bases presented or implied in Section 4 of \cite{AFW1} 
 (see Theorems~4.4~and~4.16 there)
 or in \cite{AFWgeodecomp}, which are all the same.
\end{remark}

\begin{remark}
    \label{rem:morevectoranalysis} 
    We illustrate the basis for the space of higher order Whitney forms $\calP_{r}^{-}\Lambda^{k}(T)$ 
    in the language of vector analysis. 
Over a triangle, a basis for the Raviart-Thomas space of degree $r-1$ is 
    \begin{align}
        \left\{\; 
            \lambda_T^{\alpha} \left( \lambda^T_p \nablalambda^T_q - \lambda^T_q \nablalambda^T_p \right)
            \suchthat* 
            \begin{array}{l}
                \alpha \in A(r-1,2), \; p,q \in \{0,1,2\}, 
                \\
                p < q,
                \; 
                \lfloor\alpha\rfloor \geq p
            \end{array} 
        \;\right\}.
    \end{align}
    Over a tetrahedron, 
    we have a basis for the curl-conforming N\'ed\'elec elements of the first kind of polynomial degree $r-1$,
    \begin{align}
        \left\{\; 
            \lambda_T^{\alpha} \left( \lambda^T_p \nablalambda^T_q - \lambda^T_q \nablalambda^T_p \right)
            \suchthat* 
            \begin{array}{l}
                \alpha \in A(r-1,3), \; p,q \in \{0,1,2,3\}, 
                \\
                p < q,
                \; 
                \lfloor\alpha\rfloor \geq p
            \end{array} 
        \;\right\}
        ,
    \end{align}
    and 
    we have a basis for the divergence-conforming N\'ed\'elec elements of the first kind of polynomial degree $r-1$,
    \begin{align}
        \begin{split}
        &
        \bigg\{\; 
            \lambda_T^{\alpha} \left(
                \lambda^T_p \nablalambda^T_q \times \nablalambda^T_s
                - 
                \lambda^T_q \nablalambda^T_p \times \nablalambda^T_s
                + 
                \lambda^T_s \nablalambda^T_p \times \nablalambda^T_q
                \right)
            \\&\qquad\qquad\qquad\qquad\qquad
            \qquad\qquad
            \bigg\vert
            \begin{array}{l}
                \alpha \in A(r-1,3), \; p,q,s \in \{0,1,2,3\}, 
                \\
                p < q < s,
                \; 
                \lfloor\alpha\rfloor \geq p
            \end{array} 
        \;\bigg\}.                 
        \end{split}
    \end{align}
\end{remark}

\subsection{Extension Operators} \label{subsec:finiteelementspaces:extension}

We define extension operators with desirable properties. 
Whenever $F$ is a subsimplex of $T$, we consider the operators
\begin{align*}
 \ext^{k,r}_{F,T} : \mathring\calP_r\Lambda^k(F) \rightarrow \calP_r\Lambda^k(T),
 \quad
 \ext^{k,r,-}_{F,T} : \mathring\calP^{-}_r\Lambda^k(F) \rightarrow \calP^{-}_r\Lambda^k(T),
\end{align*}
that are defined by setting
\begin{gather*}
 \ext^{k,r}_{F,T} \lambda_F^{\alpha}\cartanlambda^F_{\sigma}
 =
 \lambda_T^{ \widetilde\alpha }\cartanlambda^T_{ \widetilde\sigma }
 ,
 \quad
 \lambda_F^{\alpha}\cartanlambda^F_{\sigma} \in \calB\mathring\calP_r\Lambda^{k}(F)
 ,
 \\
 \ext^{k,r,-}_{F,T} \lambda_F^{\alpha} \phi^F_{\rho}
 =
 \lambda_T^{ \widetilde\alpha } \phi^T_{ \widetilde\rho }
 ,
 \quad
 \lambda_F^{\alpha} \phi^F_{\rho} \in \calB\mathring\calP^{-}_r\Lambda^{k}(F)
 , 
\end{gather*}
where $\widetilde\alpha = \alpha \circ \imath(F,T)^{\dagger}$ over $[\imath(F,T)]$ and zero otherwise, 
and where $\widetilde\sigma = \imath(F,T) \circ \sigma$ 
and $\widetilde\rho = \imath(F,T) \circ \rho$, respectively.
We emphasize that these extension operators, like our bases, generally depend on the ordering of the vertices. 

Their desirable properties are as follows. 
Whenever $G$ is another subsimplex of $T$ with $F \subseteq G$, then 
\begin{gather} \label{math:desirableproperty:eins}
 \trace_{T,G} \ext^{k,r}_{F,T}   = \ext^{k,r}_{F,G}
 , 
 \quad
 \trace_{T,G} \ext^{k,r,-}_{F,T} = \ext^{k,r,-}_{F,G}
 .
\end{gather}
Whenever $G$ is another subsimplex of $T$ with $F \nsubseteq G$, then 
\begin{gather} \label{math:desirableproperty:zwei}
 \trace_{T,G} \ext^{k,r}_{F,T} = 0, 
 \quad
 \trace_{T,G} \ext^{k,r,-}_{F,T} = 0.
\end{gather}
We will discuss in Section~\ref{sec:geometricdecompositions} how the two properties \eqref{math:desirableproperty:eins} and \eqref{math:desirableproperty:zwei} 
of the extension operators 
facilitate a geometric decomposition of the finite element space.
As a precursor, we notice the decompositions
\begin{gather*}
\calB\calP_{r}\Lambda^{k}(T)     = \bigcup_{F \subseteq T} \ext^{k,r}_{F,T} \calB\mathring\calP_{r}\Lambda^{k}(F), 
  \quad
  \calB\calP_{r}^{-}\Lambda^{k}(T) = \bigcup_{F \subseteq T} \ext^{k,r,-}_{F,T} \calB\mathring\calP_{r}^{-}\Lambda^{k}(F) 
 \end{gather*}
of the bases into contributions associated with the subsimplices of the simplex $T$.

\begin{remark}
  We compare these findings with the literature. 
  Our extension operator $\ext^{k,r,-}_{F,T}$ for the higher-order Whitney forms
  coincides with the extension operator in \cite{AFWgeodecomp}.
  The operator $\ext^{k,r}_{F,T}$ has appeared only implicitly in \cite{AFW1}
  and is different from the extension operator for the $\calP_r\Lambda^k$-family of spaces in \cite{AFWgeodecomp}.
  Generally, given extension operators satisfying \eqref{math:desirableproperty:eins}--\eqref{math:desirableproperty:zwei},
  any choice of bases for $\mathring\calP_{r}\Lambda^{k}(F)$ or $\mathring\calP_{r}^{-}\Lambda^{k}(F)$
  induces bases with an analogous decomposition. 
\end{remark}

\section{Linear dependencies}
\label{sec:lineardependencies}

We have introduced the canonical spanning sets for the spaces of polynomial differential forms over a simplex $T$. 
In this section we prove correspondences between the linear dependencies of those spanning sets
and discuss their applications.
Specifically, 
Theorem~\ref{prop:coefficientequivalence_eins} and Theorem~\ref{prop:coefficientequivalence_zwei} below 
show that we have a correspondence between the linear dependencies
of the canonical spanning sets of $\calP_r\Lambda^{k}(T)$ and $\mathring\calP^{-}_{r+k+1}\Lambda^{n-k}(T)$,
and a correspondence between the linear dependencies of the canonical spanning sets
of $\calP^{-}_{r+1}\Lambda^{n-k}(T)$ and $\mathring\calP_{r+n-k+1}\Lambda^{k}(T)$.
As a consequence, this defines isomorphisms between those spaces. 

\begin{remark}
    To put the main results into perspective,
    we make the following informal observation.
    Recall $\phi_{T}$, the constant unit volume form over $T$, 
    and let $\lambda_{T} := \lambda_{0} \lambda_{1} \lambda_{2} \cdots \lambda_{n}$
    be the canonical bubble function associated with the simplex $T$,
    which vanishes along the simplex boundary.
    The assignments 
    $
        \lambda^{\alpha} \phi_{T} \mapsto \lambda_{T} \lambda^{\alpha}
    $
    define isomorphisms from $\calP_{r+1}^{-}\Lambda^{n}(T)$
    onto $\calP_{r+n+1}\Lambda^{0}(T)$.
    They are stated in terms of the spanning sets 
    $\calS\calP_{r+1}^{-}\Lambda^{n}(T)$
    and $\calS\calP_{r+n+1}\Lambda^{0}(T)$,
    which are bases in this special case.
    We want to generalize this observation to barycentric differential forms.
    Since we have natural spanning sets for the spaces of polynomial differential forms, 
    we state our main results as \textit{correspondences between linear dependencies}
    of those natural spanning sets. 
    The resulting theorems, however, are more complex. 
\end{remark}

\begin{theorem} \label{prop:coefficientequivalence_eins}
    Let $r \in \bbN_{0}$ and $k \in [0:n]$. 
    Let $\omega_{\alpha\sigma} \in \bbC$ for $\sigma \in \Sigma(k,n)$ and $\alpha \in A(r,n)$.
    Then
    \begin{align}
        \label{math:coefficientequivalence_eins:partner}
        \sum_{ \substack{ \alpha \in A(r,n) \\ \sigma \in \Sigma(k,n) } }
        \omega_{\alpha\sigma} \lambda^{\alpha} \cartanlambda_{\sigma} = 0
        \quad\equivalent\quad
        \sum_{ \substack{ \alpha \in A(r,n) \\ \sigma \in \Sigma(k,n) } }
        \eps(\sigma,\sigma^c) \omega_{\alpha\sigma} \lambda^{\alpha} \lambda_{\sigma} \phi_{\sigma^c} = 0,
\end{align}
    each of which is the case if and only if
    \begin{align}
        \label{math:coefficientequivalence_eins:drittesrad}
        \omega_{\alpha\sigma}
        -
        \sum_{ p \in [\sigma] } \eps(p,\sigma-p) \omega_{\alpha,\sigma-p+0}
        =
        0
    \end{align}
    holds for $\alpha \in A(r,n)$ and 
    $\sigma \in \Sigma(k,n)$ with $0 \notin [\sigma]$.
\end{theorem}

\begin{proof}
    The statement is trivial if $k = 0$
    because in that case we merely restate that  
    $\calB\calP_{r}\Lambda^{0}(T)$ and $\calB\mathring\calP_{r+1}^{-}\Lambda^{n}(T)$
    are bases. 
    So assume that $1 \leq k \leq n$.
Consider 
\begin{align*}
        S_L 
        &:=
        \sum_{ \substack{ \alpha \in A(r,n) \\ \sigma \in \Sigma(k,n) } }
        \omega_{\alpha\sigma} \lambda^{\alpha} \cartanlambda_{\sigma}
        =
        \sum_{ \substack{ \alpha \in A(r,n) \\ \sigma \in \Sigma(k,n) \\ 0 \notin [\sigma] } }
        \omega_{\alpha\sigma} \lambda^{\alpha} \cartanlambda_{\sigma}
        +
        \sum_{ \substack{ \alpha \in A(r,n) \\ \sigma \in \Sigma(k,n) \\ 0 \in [\sigma] } }
        \omega_{\alpha\sigma} \lambda^{\alpha} \cartanlambda_{\sigma}
        .
    \end{align*}
    For $\sigma \in \Sigma(k,n)$ with $0 \in [\sigma]$ we observe 
\begin{align*}
        \cartanlambda_\sigma
        =
        \cartanlambda_0 \wedge \cartanlambda_{\sigma-0}
        &=
        - \sum_{ q \in [\sigma^c] } \cartanlambda_q \wedge \cartanlambda_{\sigma-0}
=
        \sum_{ q \in [\sigma^c] } \eps(q,\sigma) \cartanlambda_{\sigma-0+q}
        ,
    \end{align*}
    Direct application of this observation shows 
\begin{align*}
        S_L
        &=
        \sum_{ \substack{ \alpha \in A(r,n) \\ \sigma \in \Sigma(k,n) \\ 0 \notin [\sigma] } }
        \omega_{\alpha\sigma} \lambda^{\alpha} \cartanlambda_{\sigma}
        +
        \sum_{ \substack{ \alpha \in A(r,n) \\ \sigma \in \Sigma(k,n) \\ 0 \in [\sigma], \; q \in [\sigma^c] } }
        \omega_{\alpha\sigma} \lambda^{\alpha}
\eps(q,\sigma) \cartanlambda_{\sigma-0+q}
        .
    \end{align*}
    We want to reindex the second sum.
    To every pair $(\sigma,q)$ where 
    $\sigma \in \Sigma(k,n)$ with $0 \in [\sigma]$ and $q \in [\sigma^c]$
    corresponds a unique pair $(\tau,q)$
    where $\tau \in \Sigma(k,n)$ with $0 \notin [\tau]$ and $q \in [\tau]$ 
    by setting $\tau = \sigma - 0 + q$.
    Note that $\eps(q,\sigma) \cartanlambda_{\sigma-0+q} = \eps(q,\tau - q + 0) \cartanlambda_{\tau}$.
    With that reindexing it follows that 
\begin{align*}
        S_L
        &=
        \sum_{ \substack{ \alpha \in A(r,n) \\ \sigma \in \Sigma(k,n) \\ 0 \notin [\sigma] } }
        \left(
        \omega_{\alpha\sigma}
        +
        \sum_{ p \in [\sigma] } \eps(p,\sigma-p+0) \omega_{\alpha,\sigma-p+0}
        \right) \lambda^{\alpha} \cartanlambda_{\sigma}
        \\
        &=
        \sum_{ \substack{ \alpha \in A(r,n) \\ \sigma \in \Sigma(k,n) \\ 0 \notin [\sigma] } }
        \left(
        \omega_{\alpha\sigma}
        -
        \sum_{ p \in [\sigma] } \eps(p,\sigma-p) \omega_{\alpha,\sigma-p+0}
        \right) \lambda^{\alpha} \cartanlambda_{\sigma}
        .
\end{align*}
    This is an expression in a basis of $\calP_r\Lambda^{k}(T)$.
On the other hand, define $S_{R}$ by 
    \begin{align*}
        S_R
        &:= 
        \sum_{ \substack{ \alpha \in A(r,n) \\ \sigma \in \Sigma(k,n) } }
        \eps(\sigma,\sigma^c) \omega_{\alpha\sigma} \lambda^{\alpha} \lambda_{\sigma} \phi_{\sigma^c}
        \\&=
        \sum_{ \substack{ \alpha \in A(r,n) \\ \sigma \in \Sigma(k,n) \\ 0 \notin [\sigma] } }
        \eps(\sigma,\sigma^c) \omega_{\alpha\sigma} \lambda^{\alpha} \lambda_{\sigma} \phi_{\sigma^c}
        +
        \sum_{ \substack{ \alpha \in A(r,n) \\ \sigma \in \Sigma(k,n) \\ 0 \in [\sigma] } }
        \eps(\sigma,\sigma^c) \omega_{\alpha\sigma} \lambda^{\alpha} \lambda_{\sigma} \phi_{\sigma^c}
        .
    \end{align*}
Using Lemma~\ref{prop:whitneycancellationlemma}, 
    for $\sigma \in \Sigma(k,n)$ with $0 \in [\sigma]$ we observe 
    \begin{align*}
        \lambda_{\sigma} \phi_{\sigma^c}
        =
        \lambda_{\sigma-0} \lambda_0 \phi_{\sigma^c}
        &=
        - 
        \lambda_{\sigma-0} \sum_{ q \in [\sigma^c] } 
        \eps(q,\sigma^c+0-q) \lambda_{q} \phi_{\sigma^c + 0 - q}
        \\
        &=
        \lambda_{\sigma-0} \sum_{ q \in [\sigma^c] } 
        \eps(q,\sigma^c-q) \lambda_{q} \phi_{\sigma^c + 0 - q}
        .
    \end{align*}
    We apply that observation and use the same reindexing as above,
    which shows 
    \begin{align*}
        &
        \sum_{ \substack{ \alpha \in A(r,n) \\ \sigma \in \Sigma(k,n) \\ 0 \in [\sigma] } }
        \eps(\sigma,\sigma^c)
        \omega_{\alpha\sigma} \lambda^{\alpha} 
        \lambda_{\sigma} \phi_{\sigma^c}
        =
\sum_{ \substack{ \alpha \in A(r,n) \\ \sigma \in \Sigma(k,n) \\ 0 \in [\sigma] \\ q \in [\sigma^c] } }
        \eps(\sigma,\sigma^c)
        \eps(q,\sigma^c-q) 
        \omega_{\alpha\sigma} 
        \lambda^{\alpha} 
        \lambda_{\sigma-0+q} 
        \phi_{\sigma^c - q + 0}
        \\ &=
        \sum_{ \substack{ \alpha \in A(r,n) \\ \tau \in \Sigma(k,n) \\ 0 \notin [\tau], \; q \in [\tau] } }
        \eps(\tau + 0 - q,\tau^c - 0 + q)
        \eps(q,\tau^c - 0) 
        \omega_{\alpha,\tau + 0 - q} 
        \lambda^{\alpha} 
        \lambda_{\tau} 
        \phi_{\tau^c}
        .
    \end{align*}
    Note that
    for $\tau \in \Sigma(k,n)$ with $0 \notin [\tau]$ and $q \in [\tau]$
we have the combinatorial observation 
    \begin{gather*}
\eps(\tau,\tau^c)
        =
        \eps(q,\tau-q)
        (-1)
        \eps(q,\tau^c-0)
        \eps(\tau-q+0,\tau^c+q-0) 
        . 
    \end{gather*}
    The last formula can be seen as follows. 
    We want to bring the sequence $\tau$ followed by $\tau^c$ into ascending order. 
    First, we move $q \in [\tau]$ to the front, 
    and then apply $k-1$ transpositions to move $q$ in front of $\tau^c(0) = 0$. 
    We then apply $k$ transpositions to move $0 \in [\tau^c]$ to the very beginning,
    and order $q$ into the sequence $\tau^c-0$. 
    Now it remains to bring the sequence $\tau-q+0$ followed by $\tau^c+q-0$ into ascending order. 
    This together provides the above identity of signs.
    
    Applying all this,
    we see that $S_{R}$ equals 
\begin{align*}
        &
\sum_{ \substack{ \alpha \in A(r,n) \\ \sigma \in \Sigma(k,n) \\ 0 \notin [\sigma] } }
        \eps(\sigma,\sigma^c) 
        \left(
        \omega_{\alpha\sigma}
        -
        \sum_{p \in [\sigma]} 
        \eps(p,\sigma-p)
        \omega_{\alpha,\sigma-p+0} 
        \right) \lambda^\alpha \lambda_\sigma \phi_{\sigma^c}
        .
    \end{align*}
This is an expression in terms of a basis of 
    $\mathring\calP_{r+n-k+1}^{-}\Lambda^{n-k}(T)$.
    Thus $S_L = 0$ if and only if $S_R = 0$,
    which is the case if and only if \eqref{math:coefficientequivalence_eins:drittesrad} holds. 
\springerqed\end{proof}

\begin{theorem} \label{prop:coefficientequivalence_zwei}
    Let $r \in \bbN_{0}$ and $k \in [0:n]$. 
    Let $\omega_{\alpha\sigma} \in \bbC$ for $\sigma \in \Sigma(k,n)$ and $\alpha \in A(r,n)$.
    Then
    \begin{align}
        \label{math:coefficientequivalence_zwei:partner}
        \sum_{ \substack{ \alpha \in A(r,n) \\ \sigma \in \Sigma(k,n) } }
        \eps(\sigma,\sigma^c) \omega_{\alpha\sigma} \lambda^{\alpha} \phi_{\sigma^c} = 0
        \quad\equivalent\quad
        \sum_{ \substack{ \alpha \in A(r,n) \\ \sigma \in \Sigma(k,n) } }
        \omega_{\alpha\sigma} \lambda^{\alpha} \lambda_{\sigma^{c}} \cartanlambda_{\sigma} = 0
        ,
    \end{align}
    each of which is the case if and only if
    \begin{align}
        \label{math:coefficientequivalence_zwei:drittesrad}
        \omega_{\alpha\sigma} 
        -
        \sum_{ \substack{ q \in [\sigma]\cap[\alpha] } }
        \eps( \lfloor\sigma^c\rfloor, \sigma )
        \eps( q, \sigma - q ) 
        \omega_{\alpha+\lfloor\sigma^c\rfloor-q,\sigma + \lfloor\sigma^c\rfloor - q}
        =
        0
\end{align}
    holds for $\alpha \in A(r,n)$ and $\sigma \in \Sigma(k,n)$
    with $\lfloor\alpha\rfloor \geq \lfloor\sigma^c\rfloor$.
\end{theorem}

\begin{proof}
    If $r = 0$, the two sums in \eqref{math:coefficientequivalence_zwei:partner} 
    are already stated in terms of $\calB\calP_{r+1}^{-}\Lambda^{n-k}(T)$ and $\calB\mathring\calP_{r+n-k+1}\Lambda^{k}(T)$,
    respectively, 
    and \eqref{math:coefficientequivalence_zwei:drittesrad} just reduces to all coefficients vanishing.
    So it remains to study the case $r \geq 1$.
Furthermore, the statement is trivial if $k = 0$ 
    because in that case we merely restate that  
    $\calB\calP_{r+1}^{-}\Lambda^{n}(T)$ and $\calB\mathring\calP_{r+1}\Lambda^{0}(T)$
    are bases.
    So it remains to study the case $k \geq 1$.
    
    We define $S_{L}$ by setting
    \begin{align*}
        S_L
        &:= 
        \sum_{ \substack{ \alpha \in A(r,n) \\ \sigma \in \Sigma(k,n) } }
        \eps(\sigma,\sigma^c) \omega_{\alpha\sigma} \lambda^{\alpha} \phi_{\sigma^c}
        \\
        &
        =
        \sum_{ \substack{ \alpha \in A(r,n) \\ \sigma \in \Sigma(k,n) \\ \lfloor\alpha\rfloor \geq \lfloor\sigma^c\rfloor } }
        \eps(\sigma,\sigma^c) \omega_{\alpha\sigma} \lambda^{\alpha} \phi_{\sigma^c}
        +
        \sum_{ \substack{ \alpha \in A(r,n) \\ \sigma \in \Sigma(k,n) \\ \lfloor\alpha\rfloor < \lfloor\sigma^c\rfloor } }
        \eps(\sigma,\sigma^c) \omega_{\alpha\sigma} \lambda^{\alpha} \phi_{\sigma^c}.
    \end{align*}
    Using Lemma~\ref{prop:whitneycancellationlemma}, 
    for each $\sigma \in \Sigma(k,n)$ and $\alpha \in A(r,n)$
    with $\lfloor\alpha\rfloor < \lfloor\sigma^c\rfloor$ we have 
    \begin{align*}
        \lambda^{\alpha} \phi_{\sigma^c}
        =
        \lambda^{\alpha - \lfloor\alpha\rfloor} \lambda_{\lfloor\alpha\rfloor} \phi_{\sigma^c}
        &=
        -
        \lambda^{\alpha - \lfloor\alpha\rfloor}
        \sum_{ \substack{ q \in [\sigma^c] } }
        \eps( q, \sigma^c + \lfloor\alpha\rfloor - q ) \lambda_q \phi_{ \sigma^c + \lfloor\alpha\rfloor - q }
        \\&
        =
        \sum_{ \substack{ q \in [\sigma^c] } }
        \eps(q,\sigma^c-q) \lambda^{\alpha - \lfloor\alpha\rfloor + q}
        \phi_{\sigma^c + \lfloor\alpha\rfloor - q}
        .
\end{align*}
    Therefore we can rewrite 
    \begin{align*}
        \sum_{ \substack{ \alpha \in A(r,n) \\ \sigma \in \Sigma(k,n) \\ \lfloor\alpha\rfloor < \lfloor\sigma^c\rfloor } }
        \eps(\sigma,\sigma^c) \omega_{\alpha\sigma} \lambda^{\alpha} \phi_{\sigma^c}
        =
        \sum_{ \substack{ \alpha \in A(r,n) \\ \sigma \in \Sigma(k,n) \\ \lfloor\alpha\rfloor < \lfloor\sigma^c\rfloor \\ q \in [\sigma^c] } }
        \eps(\sigma,\sigma^c) \eps(q,\sigma^c-q) \omega_{\alpha\sigma} \lambda^{\alpha - \lfloor\alpha\rfloor + q}
        \phi_{\sigma^c + \lfloor\alpha\rfloor - q}
        .
\end{align*}
    We reindex the last sum. 
    Let $\sigma \in \Sigma(k,n)$, $\alpha \in A(r,n)$ 
    and $q \in [\sigma^{c}]$ with $\lfloor\alpha\rfloor < \lfloor\sigma^c\rfloor$. 
    Thus $\beta = \alpha - \lfloor\alpha\rfloor + q$
    and also $\rho = \sigma - \lfloor\alpha\rfloor + q$,
    we observe $q \in [\rho]\cap[\beta]$ and $\lfloor\rho^c\rfloor = \lfloor \sigma^c + \lfloor\alpha\rfloor - q \rfloor = \lfloor\alpha\rfloor < \lfloor\beta\rfloor$.
    Conversely,
    given $\rho \in \Sigma(k,n)$, $\beta \in A(r,n)$ 
    and $q \in [\rho] \cap [\beta]$ with $\lfloor\beta\rfloor > \lfloor\rho^c\rfloor$,
    we construct $\alpha = \beta + \lfloor\rho^c\rfloor - q$
    and $\sigma = \rho + \lfloor\rho^c\rfloor - q$.
    Thus $\lfloor\alpha\rfloor = \lfloor \beta + \lfloor\rho^c\rfloor - q \rfloor = \lfloor\rho^c\rfloor$
    and $\lfloor\sigma^c\rfloor = \lfloor \rho^c - \lfloor\rho^c\rfloor + q \rfloor > \lfloor\rho^c\rfloor$
    as well as $q \in [\sigma^c]$.
    Both constructions invert each other. 
    Hence the tuple $(\sigma,\alpha,q)$ uniquely determines the tuple $(\rho,\beta,q)$ and vice versa.
    After reindexing, 
    \begin{align*}
        &\quad
        \sum_{ \substack{ \alpha \in A(r,n) \\ \sigma \in \Sigma(k,n) \\ \lfloor\alpha\rfloor < \lfloor\sigma^c\rfloor } }
        \sum_{ \substack{ q \in [\sigma^c] } }
        \eps(\sigma,\sigma^c) \eps(q,\sigma^c-q) \omega_{\alpha\sigma} \lambda^{\alpha - \lfloor\alpha\rfloor + q}
        \phi_{\sigma^c + \lfloor\alpha\rfloor - q}
        \\&=
        \sum_{ \substack{ \beta \in A(r,n) \\ \rho \in \Sigma(k,n) \\ \lfloor\beta\rfloor \geq \lfloor\rho^c\rfloor \\ q \in [\rho]\cap[\beta] } }
        \eps(\rho+\lfloor\rho^{c}\rfloor-q,\rho^c-\lfloor\rho^{c}\rfloor+q) 
        \eps(q,\rho^c - \lfloor\rho^c\rfloor) 
        \omega_{\beta+\lfloor\rho^{c}\rfloor-q,\rho+\lfloor\rho^{c}\rfloor-q}
        \lambda^{\beta} \phi_{\rho^c}
        . 
    \end{align*}
    Consider $\rho \in \Sigma(k,n)$, $\beta \in A(r,n)$ and $q \in [\rho]\cap[\beta]$
    such that $\lfloor\beta\rfloor \geq \lfloor\rho^c\rfloor$.
    We immediately see $q > \lfloor\rho^c\rfloor$,
    and we make the combinatorial observations 
    \begin{gather*}
\eps(\lfloor\rho^{c}\rfloor,\rho-q)
        =
        \eps(\lfloor\rho^{c}\rfloor,\rho),
        \\
        \eps(\rho,\rho^c) 
        =
        -
        \eps(\rho-q,q) 
        \eps(\rho-q,\lfloor\rho^{c}\rfloor) 
        \eps(q,\rho^c - \lfloor\rho^c\rfloor) 
        \eps(\rho+\lfloor\rho^{c}\rfloor-q,\rho^c-\lfloor\rho^{c}\rfloor+q) 
        .
    \end{gather*}
    The last identity of signs is derived as follows. 
    We want to order the sequence $\rho$ followed by $\rho^c$ in ascending order.
    For that, we first move $q \in [\rho]$ to the end of the first part. 
    Then we switch the position of $q$ and $\lfloor\rho^{c}\rfloor$.
    Now we order the sequence $\rho - q$ followed by $\lfloor\rho^{c}\rfloor$
    and order $q$ into $\rho - \lfloor\rho^{c}\rfloor$.
    So it remains to order $\rho+\lfloor\rho^{c}\rfloor-q$ and $\rho^c-\lfloor\rho^{c}\rfloor+q$.
    
    The combination of those steps shows that $S_L$ equals
    {\begin{align*}
        \sum_{ \substack{ \alpha \in A(r,n) \\ \sigma \in \Sigma(k,n) \\ \lfloor\alpha\rfloor \geq \lfloor\sigma^c\rfloor } }
        \eps(\sigma,\sigma^c)
        \left(
            \omega_{\alpha\sigma}
            -
            \sum_{ \substack{ q \in [\sigma]\cap[\alpha] } }
            \eps(q,\sigma-q) 
            \eps(\lfloor\rho^{c}\rfloor,\sigma) 
            \omega_{\alpha+\lfloor\sigma^{c}\rfloor-q,\sigma+\lfloor\sigma^{c}\rfloor-q}
        \right)
        \lambda^{\alpha} \phi_{\sigma^c}
        .
    \end{align*}
    }This an expression in terms of a basis of $\calP_{r+1}^{-}\Lambda^{n-k}(T)$.
    Now, define $S_{R}$ by 
    \begin{align*}
        S_R 
        &:=
        \sum_{ \substack{ \alpha \in A(r,n) \\ \sigma \in \Sigma(k,n) } }
        \omega_{\alpha\sigma} \lambda^{\alpha} \lambda_{\sigma^c} \cartanlambda_{\sigma}
        \\
        &
        =
        \sum_{ \substack{ \alpha \in A(r,n) \\ \sigma \in \Sigma(k,n) \\ \lfloor\alpha\rfloor \geq \lfloor\sigma^c\rfloor } }
        \omega_{\alpha\sigma} \lambda^{\alpha} \lambda_{\sigma^c} \cartanlambda_{\sigma}
        +
        \sum_{ \substack{ \alpha \in A(r,n) \\ \sigma \in \Sigma(k,n) \\ \lfloor\alpha\rfloor < \lfloor\sigma^c\rfloor } }
        \omega_{\alpha\sigma} \lambda^{\alpha} \lambda_{\sigma^c} \cartanlambda_{\sigma}
        .
    \end{align*}
    For any $\alpha \in A(r,n)$ and $\sigma \in \Sigma(k,n)$
    with $\lfloor\alpha\rfloor < \lfloor\sigma^c\rfloor$ we see $\lfloor\alpha\rfloor \in [\sigma]$ and so 
    \begin{align*}
        \lambda^{\alpha} \lambda_{\sigma^c} \cartanlambda_{\sigma}
        &=
        \lambda^{\alpha} \lambda_{\sigma^c} 
        \eps( \lfloor\alpha\rfloor, \sigma - \lfloor\alpha\rfloor )
        \cartanlambda_{\lfloor\alpha\rfloor} 
        \wedge \cartanlambda_{\sigma - \lfloor\alpha\rfloor}
        \\&=
        -
        \lambda^{\alpha} \lambda_{\sigma^c} 
        \eps( \lfloor\alpha\rfloor, \sigma - \lfloor\alpha\rfloor )
        \sum_{ q \in [\sigma^c] }
        \eps( q, \sigma - \lfloor\alpha\rfloor ) 
        \cartanlambda_{\sigma - \lfloor\alpha\rfloor + q}
        .
    \end{align*}
    Here, we have used \eqref{math:partitionofzero}. 
    Thus we find 
    \begin{align*}
        &
        \sum_{ \substack{ \alpha \in A(r,n) \\ \sigma \in \Sigma(k,n) \\ \lfloor\alpha\rfloor < \lfloor\sigma^c\rfloor } }
        \omega_{\alpha\sigma} \lambda^{\alpha} \lambda_{\sigma^c} \cartanlambda_{\sigma}
=
\sum_{ \substack{ \alpha \in A(r,n) \\ \sigma \in \Sigma(k,n) \\ \lfloor\alpha\rfloor < \lfloor\sigma^c\rfloor \\ q \in [\sigma^c] } }
        \omega_{\alpha\sigma} \lambda^{\alpha} \lambda_{\sigma^c} 
        \eps( \lfloor\alpha\rfloor, \sigma - \lfloor\alpha\rfloor )
        \eps( q, \sigma ) 
        \cartanlambda_{\sigma - \lfloor\alpha\rfloor + q}
        .
\end{align*}
    Using the same reindexing as previously in this proof,
    we find that the last sum equals 
    \begin{align*}
        &
        \sum_{ \substack{ \beta \in A(r,n) \\ \rho \in \Sigma(k,n) \\ \lfloor\beta\rfloor \geq \lfloor\rho^c\rfloor \\ q \in [\rho]\cap[\beta] } }
\omega_{\beta+\lfloor\rho^c\rfloor-q,\rho + \lfloor\rho^c\rfloor - q}
        \lambda^{\beta+\lfloor\rho^c\rfloor-q} \lambda_{\rho^c-\lfloor\rho^c\rfloor+q} 
        \eps( \lfloor\rho^c\rfloor, \rho - q )
        \eps( q, \rho + \lfloor\rho^c\rfloor - q ) 
        \cartanlambda_{\rho}
.
    \end{align*}
    For each $\beta \in A(r,n)$, $\rho \in \Sigma(k,n)$, and $q \in [\rho]\cap[\beta]$ with $\lfloor\beta\rfloor \geq \lfloor\rho^c\rfloor$
    we have $q > \lfloor\rho^c\rfloor$,
    as noted earlier,
    and so we can use that 
    \begin{gather*}
        \lambda^{\beta+\lfloor\rho^c\rfloor-q} \lambda_{\rho^c-\lfloor\rho^c\rfloor+q} = \lambda^{\beta} \lambda_{\rho^c},
        \\
        \eps( \lfloor\rho^c\rfloor, \rho - q ) = \eps( \lfloor\rho^c\rfloor, \rho ),
        \quad 
        \eps( q, \rho + \lfloor\rho^c\rfloor - q ) = - \eps( q, \rho - q ).
    \end{gather*}
    Putting this all together, 
    we see that $S_{R}$ equals 
{\begin{align*}
\sum_{ \substack{ \alpha \in A(r,n) \\ \sigma \in \Sigma(k,n) \\
        \lfloor\alpha\rfloor \geq \lfloor\sigma^c\rfloor } }
        \left(
        \omega_{\alpha\sigma} 
        -
        \sum_{ \substack{ q \in [\sigma]\cap[\alpha] } }
        \eps( \lfloor\sigma^c\rfloor, \sigma )
        \eps( q, \sigma - q ) 
        \omega_{\alpha+\lfloor\sigma^c\rfloor-q,\sigma + \lfloor\sigma^c\rfloor - q}
\right)
        \lambda^{\alpha} \lambda_{\sigma^c} \cartanlambda_{\sigma}
        .
\end{align*}
    }This is an expression in terms of a basis of $\mathring\calP_{r+n-k+1}\Lambda^{k}(T)$.
Thus $S_L = 0$ if and only if $S_R = 0$,
    which is the case if and only if \eqref{math:coefficientequivalence_zwei:drittesrad} holds. 
\springerqed\end{proof}

The two results above have multiple applications. 
For example, any basis in one space of the pairing corresponds to a basis in the other space. 
Another consequence is that we have isomorphisms 
\begin{align*}
 \calP_r\Lambda^{k}(T)
 \simeq 
 \mathring\calP_{r+k+1}^{-}\Lambda^{n-k}(T), 
 \quad
 \calP_{r+1}^{-}\Lambda^{n-k}(T)
 \simeq
 \mathring\calP_{r+n-k+1}\Lambda^{k}(T).
\end{align*}
that preserve the canonical spanning sets.
These isomorphisms are explicitly described as follows. 
We have a linear isomorphism 
from $\calP_r\Lambda^{k}(T)$ to $\mathring\calP^{-}_{r+k+1}\Lambda^{n-k}(T)$
that in terms of spanning sets can be written as 
\begin{align}
  \label{math:coefficientisomorphism_eins}
  \sum_{ \substack{ \alpha \in A(r,n) \\ \sigma \in \Sigma(k,n) } }
  \omega_{\alpha\sigma} \lambda^{\alpha} \cartanlambda_{\sigma}
  \mapsto
  \sum_{ \substack{ \alpha \in A(r,n) \\ \sigma \in \Sigma(k,n) } }
  \omega_{\alpha\sigma} \lambda^{\alpha} \lambda_{\sigma} \phi_{\sigma^c}
  ,
\end{align}
and we have a linear isomorphism 
from $\calP^{-}_{r+1}\Lambda^{n-k}(T)$ to $\mathring\calP_{r+n-k+1}\Lambda^{k}(T)$ 
that in terms of spanning sets can be written as 
\begin{align}
  \label{math:coefficientisomorphism_zwei}
  \sum_{ \substack{ \alpha \in A(r,n) \\ \sigma \in \Sigma(k,n) } }
  \omega_{\alpha\sigma} \lambda^{\alpha} \phi_{\sigma^c}
  \mapsto
  \sum_{ \substack{ \alpha \in A(r,n) \\ \sigma \in \Sigma(k,n) } }
  \omega_{\alpha\sigma} \lambda^{\alpha} \lambda_{\sigma^c} \cartanlambda_{\sigma}
  .
\end{align}

\begin{remark}
    \label{rem:specialcasesoflineardependencies}
    We illustrate a few special cases. 
    In the special case $k=0$, 
    the set $\Sigma(k,n)$ only contains the empty mapping,
    Theorem~\ref{prop:coefficientequivalence_eins} and Theorem~\ref{prop:coefficientequivalence_zwei}
    translate to
\begin{align*}
        \sum_{ \substack{ \alpha \in A(r,n) } }
        \omega_{\alpha} \lambda^{\alpha} = 0
        \quad&\equivalent\quad
        \sum_{ \substack{ \alpha \in A(r,n) } }
        \omega_{\alpha} \lambda^{\alpha} \phi_{T} = 0,
        \\
        \sum_{ \substack{ \alpha \in A(r,n) } }
        \omega_{\alpha} \lambda^{\alpha} \phi_{T} = 0
        \quad&\equivalent\quad
        \sum_{ \substack{ \alpha \in A(r,n) } }
        \omega_{\alpha} \lambda^{\alpha} \lambda_{T} = 0,
    \end{align*}
    respectively, and each of those conditions is equivalent to $\omega_{\alpha} = 0$ for all $\alpha \in A(r,n)$. These correspond to the isomorphisms 
    $\calP_{r}\Lambda^{0}(T) \simeq \mathring\calP_{r+1}\Lambda^{n}(T)$
    and $\calP_{r}^{-}\Lambda^{n}(T) \simeq \mathring\calP_{r+n+1}\Lambda^{0}(T)$.
On the other hand, for discussing the special case $k=n$, 
    let us first assume that $\sigma \in \Sigma(n,n)$. 
    There exists a unique index $i \in [0:n]$ with $i \notin [\sigma]$.
    We write $\cartanlambda_{n:i} := \cartanlambda_{\sigma}$ and $\omega_{\alpha\sigma} = \omega_{\alpha,i}$.
    From definitions we easily see  
    $\phi_{\sigma^{c}} = \lambda_{i}$ and 
    $\cartanlambda_{\sigma} = \cartanlambda_{0} \wedge\cdots\wedge \cartanlambda_{i-1}  \wedge\cartanlambda_{i+1} \wedge\dots\wedge  \cartanlambda_{n}$. 
    Theorem~\ref{prop:coefficientequivalence_eins} and Theorem~\ref{prop:coefficientequivalence_zwei} translate to 
    \begin{align*}
        \sum_{ \substack{ \alpha \in A(r,n) \\ 0 \leq i \leq n } }
        \omega_{\alpha,i} \lambda^{\alpha} \cartanlambda_{n:i} = 0
        \quad&\equivalent\quad
        \sum_{ \substack{ \alpha \in A(r,n) \\ 0 \leq i \leq n } }
        (-1)^{n+i} \omega_{\alpha,i} \lambda^{\alpha} \lambda_{T} = 0,
        \\
        \sum_{ \substack{ \alpha \in A(r,n) \\ 0 \leq i \leq n } }
        (-1)^{n+i} \omega_{\alpha,i} \lambda^{\alpha} \lambda_{i} = 0
        \quad&\equivalent\quad
        \sum_{ \substack{ \alpha \in A(r,n) \\ 0 \leq i \leq n } }
        \omega_{\alpha,i} \lambda^{\alpha} \lambda_{i} \cartanlambda_{n:i} = 0
        .
    \end{align*}
    Those are relations between the spanning sets
    $\calS\calP_{r}\Lambda^{n}$ and $\calS\mathring\calP_{r+n+1}^{-}\Lambda^{0}$ 
    and between the spanning sets $\calS\calP_{r+1}^{-}\Lambda^{0}$ and $\calS\calP_{r+1}\Lambda^{n}$, 
    respectively. In the latter case, those spanning sets are bases. 
\end{remark}

Above, 
we have stated conditions on the coefficients under which linear combinations 
of differential forms from the canonical spanning equal zero. 
We prove two more auxiliary results,
Lemma~\ref{prop:coefficientcondition_eins} and Lemma~\ref{prop:coefficientcondition_zwei},
which further characterize the conditions on the coefficients. 
These are important in the next section.

\begin{lemma} \label{prop:coefficientcondition_eins}
 Let $r \in \bbN_{0}$ and $k \in [0:n]$.
 Let $\omega_{\alpha\sigma}$ be a family of complex numbers
 indexed over $\alpha \in A(r,n)$ and 
 $\sigma \in \Sigma(k,n)$.
 Then 
 \begin{align} \label{prop:coefficientcondition_eins:uno}
  \omega_{\alpha\sigma}
  -
  \sum_{ p \in [\sigma] } \eps(p,\sigma-p) \omega_{\alpha,\sigma-p+0}
  =
  0
 \end{align}
 holds for all $\alpha \in A(r,n)$ and 
 $\sigma \in \Sigma(k,n)$ with $0 \notin [\sigma]$
 if and only if
 \begin{align} \label{prop:coefficientcondition_eins:duo}
  \sum_{ p \in [\theta] } \eps(p,\theta-p) \omega_{\alpha,\theta-p} = 0
 \end{align}
 holds for all $\alpha \in A(r,n)$ and 
 $\theta \in \Sigma(k+1,n)$.
\end{lemma}

\begin{proof}
 The lemma is trivial in the special case $k=0$.
 So assume that $1 \leq k \leq n$.
 Clearly, \eqref{prop:coefficientcondition_eins:duo} implies \eqref{prop:coefficientcondition_eins:uno}
 via $\theta = \sigma + 0$.
 So let us suppose \eqref{prop:coefficientcondition_eins:uno} holds.
Then \eqref{prop:coefficientcondition_eins:duo} clearly holds for all $\theta$ with $0 \in [\theta]$.
 If instead $0 \notin [\theta]$, 
 then \eqref{prop:coefficientcondition_eins:uno} implies 
 \begin{align*}
  \sum_{ p \in [\theta] }
  \eps(p,\theta-p) \omega_{\alpha,\theta-p}
  &=
  \sum_{ p \in [\theta] }
  \sum_{ s \in [\theta-p] }
  \eps(p,\theta-p)
  \eps(s,\theta-p-s)
  \omega_{\alpha,\theta-p-s+0}
  \\&=
  \sum_{ p \in [\theta] }
  \sum_{ s \in [\theta-p] }
  \eps(p,s)
  \eps(p,\theta-p)
  \eps(s,\theta-s)
  \omega_{\alpha,\theta-p-s+0}
  .
 \end{align*}
 This sum vanishes by an antisymmetry argument.
 The lemma follows.
\springerqed\end{proof}

\begin{lemma} \label{prop:coefficientcondition_zwei}
 Let $r \in \bbN_{0}$ and $k \in [0:n]$.
 Let $\omega_{\alpha\sigma}$ be a family of complex numbers
 indexed over $\alpha \in A(r,n)$ and 
 $\sigma \in \Sigma(k,n)$.
 Then 
 \begin{align} \label{prop:coefficientcondition_zwei:uno}
  \omega_{\alpha\sigma} 
  -
\sum_{ \substack{ p \in [\sigma]\cap[\alpha] } }
  \eps( \lfloor\sigma^c\rfloor, \sigma )
  \eps( p, \sigma - p ) 
  \omega_{\alpha + \lfloor\sigma^c\rfloor - p,\sigma + \lfloor\sigma^c\rfloor - p}
  =
  0
\end{align}
 holds for all $\alpha \in A(r,n)$ and 
 $\sigma \in \Sigma(k,n)$ with $\lfloor\alpha\rfloor \geq \lfloor\sigma^c\rfloor$
 if and only if
 \begin{align} \label{prop:coefficientcondition_zwei:duo}
  \sum_{p \in [\theta] \cap [\beta]}
  \eps(p,\theta-p)
  \omega_{\beta-p,\theta-p}
  =
  0
 \end{align}
 holds for all $\beta \in A(r+1,n)$ and 
 $\theta \in \Sigma(k+1,n)$.
\end{lemma}

\begin{proof}
    The lemma is trivial in the special cases $k=0$ or $r=0$,
    so we assume that $1 \leq k \leq n$ and $r \geq 1$.
Suppose \eqref{prop:coefficientcondition_zwei:duo} holds.
    Let $\alpha \in A(r,n)$ and $\sigma \in \Sigma(k,n)$
    satisfy $\lfloor\alpha\rfloor \geq \lfloor\sigma^c\rfloor$. 
    Using \eqref{prop:coefficientcondition_zwei:duo} with $\beta = \alpha + \lfloor\sigma^c\rfloor$ 
    and $\theta = \sigma + \lfloor\sigma^c\rfloor$, we get 
    \begin{align*}
        0
        &=
        \sum_{ \substack{ p \in [\sigma+\lfloor\sigma^c\rfloor]\cap[\alpha+\lfloor\sigma^c\rfloor] } }
        \eps( p, \sigma + \lfloor\sigma^c\rfloor - p ) 
        \omega_{\alpha+\lfloor\sigma^c\rfloor-p,\sigma + \lfloor\sigma^c\rfloor - p}
        \\
        &=
        \eps( \lfloor\sigma^c\rfloor, \sigma )
        \omega_{\alpha\sigma} 
        +
        \sum_{ \substack{ p \in [\sigma]\cap[\alpha] } }
        \eps( p, \sigma + \lfloor\sigma^c\rfloor - p ) 
        \omega_{ \alpha + \lfloor\sigma^c\rfloor - p, \sigma + \lfloor\sigma^c\rfloor - p }
        .
    \end{align*}
    We have $\eps( p, \sigma + \lfloor\sigma^c\rfloor - p ) = - \eps( p, \sigma - p )$
    because $\lfloor\alpha\rfloor \geq \lfloor\sigma^c\rfloor$,
    and so \eqref{prop:coefficientcondition_zwei:uno} follows. 
Conversely, 
    we suppose that \eqref{prop:coefficientcondition_zwei:uno} holds and derive \eqref{prop:coefficientcondition_zwei:duo}.
    Let $\beta \in A(r+1,n)$ and $\theta \in \Sigma(k+1,n)$.
    We make a case distinction. In the case $0 \in [\theta] \cap [\beta]$,
    we set $\sigma = \theta - 0$ and $\alpha = \beta - 0$,
    noting that $\lfloor\sigma^c\rfloor = 0$.
    Via those definitions, 
    \begin{align*}
        \sum_{p \in [\theta] \cap [\beta]}
        \eps(p,\theta-p)
        \omega_{\beta-p,\theta-p}
        =
        \eps( 0, \sigma ) \omega_{\alpha\sigma} 
        +
        \sum_{ \substack{ p \in [\sigma]\cap[\alpha] } }
        \eps( p, \sigma + 0 - p ) 
        \omega_{\alpha+0-p,\sigma + 0 - p}
        .
    \end{align*}
    We use $\eps( 0, \sigma ) = 1$ and $\eps( p, \sigma + 0 - p ) = - \eps( p, \sigma - p )$,
    and \eqref{prop:coefficientcondition_zwei:uno} to derive \eqref{prop:coefficientcondition_zwei:duo}. 
    It remains to consider the case $0 \notin [\theta] \cap [\beta]$.
    For such $\theta$ and $\beta$, 
    we reuse the results from the first case and get 
\begin{align*}
        &\quad
        \sum_{ q \in [\theta]\cap[\beta] }
        \eps( q,\theta-q) \omega_{\beta-q,\theta-q}
\\&=
        \sum_{ q \in [\theta]\cap[\beta] }
        \sum_{ p \in [\theta]\cap[\beta] \setminus \{q\} }
        \eps( q,\theta-q) \eps(p,\theta-q-p) \omega_{\beta-q+0-p,\theta-q+0-p}
        . 
    \end{align*}
We use $\eps(p,\theta-q-p) = \eps(q,p) \eps(p,\theta-p)$
    and notice that the sum vanishes if and only if
    \begin{align*}
        0
        =
        \sum_{ \substack{ p,q \in [\theta]\cap[\beta] \\ p \neq q } }
        \eps( \theta+0-q,q) \eps(p,q) \eps(p,\theta+0-p) 
        \omega_{\beta-q+0-p,\theta-q+0-p}
        .
\end{align*}
    Evidently, the terms in that expression cancel.
    The statement is proven.
\springerqed\end{proof}

\begin{remark}
    \label{rem:lineardependencies:lit}
    We put this section's findings into the context of the literature. 
This section completes partial results that have appeared previously. 
Our isomorphisms \eqref{math:coefficientisomorphism_eins} and \eqref{math:coefficientisomorphism_zwei}
    are identical to the isomorphism used in Theorem~4.16 and Theorem~4.22, respectively, of \cite{AFW1}. 
    However, the isomorphisms are only stated in terms of basis forms in that reference;
    we have complemented that by showing that they preserve the canonical spanning sets. 

    Identity~\eqref{math:coefficientequivalence_eins:partner} in Theorem~\ref{prop:coefficientequivalence_eins}
    is implied by Proposition~3.7 of \cite{christiansen2013high}
    but our analogous identity in Theorem~\ref{prop:coefficientequivalence_zwei} is a new result. 
    We have discussed several equations that describe the linear dependencies of the canonical spanning set, 
    Equations \eqref{math:coefficientequivalence_eins:drittesrad} and \eqref{math:coefficientequivalence_zwei:drittesrad},
    and Lemmas \ref{prop:coefficientcondition_eins} and \ref{prop:coefficientcondition_zwei}, 
    which have not appeared in previous works. 
\end{remark}

\section{Duality Pairings}
\label{sec:duality}

We have seen in the last section that there exist isomorphisms
\begin{align*}
 \calP_r\Lambda^{k}(T) \simeq \mathring\calP^-_{r+k+1}\Lambda^{n-k}(T),
 \quad
 \calP^-_{r+1}\Lambda^{n-k}(T) \simeq \mathring\calP_{r+n-k+1}\Lambda^{k}(T).
\end{align*}
In this section, we extend that observation and introduce non-degenerate bilinear pairings 
between the spaces $\calP_r\Lambda^{k}(T)$ and $\mathring\calP^-_{r+k+1}\Lambda^{n-k}(T)$ 
and between the spaces $\calP^-_{r+1}\Lambda^{n-k}(T)$ and $\mathring\calP_{r+n-k+1}\Lambda^{k}(T)$,
respectively.

Towards the construction of those bilinear pairings, 
we first introduce bilinear forms on the corresponding coefficient spaces.
We write 
\begin{gather*}
 \calP(k,r,n) := \bbC^{A(r,n) \times \Sigma(k,n)}
\end{gather*}
for the abstract complex vector space generated by the set ${A(r,n) \times \Sigma(k,n)}$.
The members of that vector space represent the coefficients 
in linear combinations of the canonical spanning sets. 

We have a bilinear form over $\calP(k,r,n)$
which for $\omega, \eta \in \calP(k,r,n)$ is given by 
\begin{align}
 \label{math:coefficientpairing:eins}
 (\omega,\eta)
 \mapsto
 \sum_{\alpha,\beta \in A(r,n)}
 \sum_{\sigma,\rho \in \Sigma(k,n)}
 \int_{T}
 \omega_{\alpha\sigma} \lambda^{\alpha} \cartanlambda_{\sigma}
 \wedge
 \eps(\rho,\rho^{c}) 
 \overline{\eta_{\beta\rho}} \lambda^{\beta} \lambda_{\rho} \phi_{\rho^{c}}
 ,
\end{align}
and another bilinear form over $\calP(k,r,n)$
which for $\omega, \eta \in \calP(k,r,n)$ is given by 
\begin{align}
 \label{math:coefficientpairing:zwei}
 (\omega,\eta)
 \mapsto
 \sum_{\alpha,\beta \in A(r,n)}
 \sum_{\sigma,\rho \in \Sigma(k,n)}
 \int_{T}
 \omega_{\alpha\sigma} \lambda^{\alpha} \lambda_{\sigma^{c}} \cartanlambda_{\sigma}
 \wedge
 \eps(\rho,\rho^{c}) 
 \overline{\eta_{\beta\rho}} \lambda^{\beta} \phi_{\rho^{c}}
 . 
\end{align}
The agenda of this section is to show that these pairings are Hermitian 
and that their degeneracy spaces correspond to the linear dependencies 
discussed in the previous section. This will show that these pairings, 
first defined over coefficients for the canonical spanning sets,
actually correspond to non-degenerate bilinear pairings 
of finite element spaces.

Basic properties of these bilinear pairings will be shown next. 
The following first lemma tells us that the bilinear pairings are,
in a certain sense, sparse. 

\begin{lemma} \label{prop:pairing_internal_sparsity}
 Let $r \in \bbN_{0}$, $k \in [0:n]$,
 and $\sigma, \rho \in \Sigma(k,n)$.
 If $\lvert[\sigma] \cap [\rho^c]\rvert > 1$, then 
 \begin{align*}
  \cartanlambda_{\sigma} \wedge \eps(\rho,\rho^{c}) \phi_{\rho^c} = 0
  .
 \end{align*}
\end{lemma}

\begin{proof}
 We expand the Whitney form $\phi_{\rho^c}$ according to \eqref{math:definitionwhitneyform}. Then 
 \begin{align*}
  \cartanlambda_{\sigma} \wedge \phi_{\rho^c} 
  = 
  \sum_{p \in [\rho^c]} \eps(p,\rho^c-p) 
  \lambda^{T}_{p} 
  \cartanlambda_{\sigma} 
  \wedge 
  \cartanlambda^{T}_{\rho^c-p}
  =
  0
\end{align*}
because of $\lvert[\sigma] \cap [\rho^c]\rvert > 1$ and the properties of the alternating product.
\springerqed\end{proof}

Next we turn our attention to the ``diagonal`` terms in the bilinear pairings.

\begin{lemma} \label{prop:pairing_internal_diagonal}
 Let $r \in \bbN_{0}$, $k \in [0:n]$,
 and $\sigma, \rho \in \Sigma(k,n)$.
 If $\lvert [\sigma] \cap [\rho^c] \rvert = 0$, then $\sigma = \rho$ and 
 \begin{align*}
  \cartanlambda_{\sigma} \wedge \eps(\sigma,\sigma^{c}) \phi_{\sigma^c}
  =
  (-1)^k \sum_{q \in [\sigma^c]} \lambda_{q} \phi_T
  . \end{align*}
\end{lemma}

\begin{proof}
 Suppose that $[\sigma]\cap[\rho^c] = \emptyset$. 
 It is easily seen that $\sigma = \rho$.
 Using \eqref{math:definitionwhitneyform}, 
 \eqref{math:barycentricdifferentialscombinatorics} and Lemma~\ref{prop:differentialdecomposition}, 
 we find 
 \begin{align*}
  \cartanlambda_{\sigma} \wedge \phi_{\sigma^c}
  &=
  \cartanlambda_{\sigma} \wedge \sum_{q \in [\sigma^c]} \lambda_{q} \eps(q,\sigma^c-q) \cartanlambda_{\sigma^c-q}
  \\&=
  \sum_{q \in [\sigma^c]} \lambda_{q} \eps(q,\sigma^c-q) \eps(\sigma,\sigma^c-q) \cartanlambda_{\sigma+\sigma^c-q}
  \\&=
  \sum_{q \in [\sigma^c]} \lambda_{q} \eps(q,\sigma^c-q) \eps(\sigma,\sigma^c-q) \eps(q,\sigma+\sigma^c-q) \phi_{T}
  .
 \end{align*} We finish by applying the combinatorial identity 
 \begin{align*}
  \eps(\sigma,\sigma^c)
  =
  \eps(q,\sigma^c-q) (-1)^k \eps(\sigma,\sigma^c-q) \eps(q,\sigma+\sigma^c-q)
  ,
 \end{align*}
 which can be seen as follows.
 Suppose we want to order the sequence of numbers given by $\sigma$ and then $\sigma^c$.
 To do so, we first move $q \in [\sigma^c]$ in front of the $\sigma^c$ part,
 then apply further $k$ transpositions to move $q$ in front of all numbers.
 Then we order the sequence given by $\sigma$ and $\sigma^c-q$,
 and finally we move $q$ into the sequence given by $\sigma+\sigma^c-q$.
 This shows the identity and finishes the proof.
\springerqed\end{proof}

Finally, the following applies to the non-zero ``off-diagonal`` terms in the bilinear pairings.

\begin{lemma} \label{prop:pairing_internal_offdiagonal}
 Let $r \in \bbN_{0}$, $k \in [0:n]$,
 and $\sigma, \rho \in \Sigma(k,n)$.
 If $\lvert [\sigma] \cap [\rho^c] \rvert = 1$, 
 then there exist unique $q \in [\sigma^c]$ and $p \in [\sigma]$
 satisfying $\rho = \sigma - p + q$,
 and we have 
 \begin{gather}
    \label{math:pairing_internal_offdiagonal:explicit}
    \cartanlambda_{\sigma} \wedge \eps(\rho,\rho^c) \phi_{\rho^c}
    =
    (-1)^{k+1} 
    \eps(p,\sigma-p) \eps(q,\sigma-p)
    \lambda_{p} \phi_{T}
    ,
    \\
    \label{math:pairing_internal_offdiagonal:symmetry_eins}
    \cartanlambda_{\sigma} \wedge \eps(\rho,\rho^c) \lambda_\rho \phi_{\rho^c}
    =
    \cartanlambda_{\rho} \wedge \eps(\sigma,\sigma^c) \lambda_\sigma \phi_{\sigma^c}
    ,
    \\
    \label{math:pairing_internal_offdiagonal:symmetry_zwei}
    \lambda_{\sigma^c} \cartanlambda_{\sigma} \wedge \eps(\rho,\rho^c) \phi_{\rho^c}
    =
    \lambda_{\rho^c} \cartanlambda_{\rho} \wedge \eps(\sigma,\sigma^c) \phi_{\sigma^c}
    .
 \end{gather}
\end{lemma}

\begin{proof}
 There exists a unique $p \in [\sigma] \cap [\rho^{c}]$.
 Since $[\sigma-p] \subseteq [\rho]$,
 there exists a unique $q \in [\rho]$
 with $\sigma - p + q = \rho$.
 We notice that $\rho^c = \sigma^c - q + p$
 as well as $[\sigma] \cap [\rho^c] = \{p\}$ and $[\sigma^c] \cap [\rho] = \{q\}$.
Expanding the Whitney form via \eqref{math:definitionwhitneyform} gives
 \begin{align*}
  \cartanlambda_{\sigma} \wedge \phi_{\rho^c}
  =
  \cartanlambda_{\sigma} \wedge \phi_{\sigma^c - q + p}
  =
  \eps(p,\sigma^c - q) \lambda_{p} \cartanlambda_{\sigma} \wedge \cartanlambda_{\sigma^c - q}
  .
 \end{align*}
 Application of Lemma~\ref{prop:differentialdecomposition} shows
 \begin{align*}
  \cartanlambda_{\sigma} \wedge \cartanlambda_{\sigma^c - q}
  =
  \eps(\sigma,\sigma^c-q) \cartanlambda_{\sigma+\sigma^c - q}
  =
  \eps(\sigma,\sigma^c-q) \eps(q,\sigma+\sigma^c-q) \phi_{T}
  .
 \end{align*}
 Applying the same combinatorial identity as in the previous proof, we see 
 \begin{align*}
  \cartanlambda_{\sigma} \wedge \phi_{\rho^c}
  &=
  (-1)^{k} \eps(p,\sigma^c - q) \eps(\sigma,\sigma^c) \eps(q,\sigma^c-q)
  \lambda_{p} \phi_{T}
  .
 \end{align*}
 We apply another combinatorial identity,
 \begin{align*}
  &
  \eps(\sigma,\sigma^c)
  =
  -
  \eps(\sigma-p,p) \eps(q,\sigma^c-q) \eps(\sigma-p,q) \eps(p,\sigma^c-q) \eps(\sigma-p+q,\sigma^c-q+p)
  ,
 \end{align*}
 which can be seen as follows. 
 We want to order the sequence $\sigma$ followed by $\sigma^c$.
 First we move $p \in [\sigma]$ to the end of the first sequence 
 and $q \in [\sigma^c]$ to the front of the second.
 Then we switch $p$ and $q$.  
 We order the sequence $\sigma-p$ followed by $q$ 
 and order $p$ into the sequence $\sigma^c-q$.
 Thus it remains to order the sequences $\sigma-p+q$ and $\sigma^c-q+p$.
We derive
 \begin{align*}
  \cartanlambda_{\sigma} \wedge \phi_{\rho^c}
  &=
  (-1)^{k+1} 
  \eps(\sigma-p,p) \eps(\sigma-p,q) \eps(\sigma-p+q,\sigma^c-q+p)
  \lambda_{p} \phi_{T}
  \\&=
  (-1)^{k+1} 
  \eps(p,\sigma-p) \eps(q,\sigma-p) \eps(\sigma-p+q,\sigma^c-q+p)
  \lambda_{p} \phi_{T}
  .
 \end{align*}
 This, together with $\rho = \sigma - p + q$, 
 shows \eqref{math:pairing_internal_offdiagonal:explicit}.
 Note that $\rho + p = \sigma + q$ and $\rho - q = \sigma - p$. 
 We show \eqref{math:pairing_internal_offdiagonal:symmetry_eins} via
 \begin{align*}
  \cartanlambda_{\sigma} \wedge \eps(\rho,\rho^c) \lambda_{\rho} \phi_{\rho^c}
  &=
  (-1)^{k+1} \eps(p,\sigma-p) \eps(q,\sigma-p)
  \lambda_{\rho} \lambda_{p} \phi_{T}
  ,
  \\
  \cartanlambda_{\rho} \wedge \eps(\sigma,\sigma^c) \lambda_{\sigma} \phi_{\sigma^c}
  &=
  (-1)^{k+1} \eps(q,\rho-q) \eps(p,\rho-q)
  \lambda_{\sigma} \lambda_{q} \phi_{T}
  .
 \end{align*}
 Note also that $\rho^c + q = \sigma^c + p$. 
 We show \eqref{math:pairing_internal_offdiagonal:symmetry_zwei} via
 \begin{align*}
  \lambda_{\sigma^c} \cartanlambda_{\sigma} \wedge \eps(\rho,\rho^c) \phi_{\rho^c}
  &=
  (-1)^{k+1} \eps(p,\sigma-p) \eps(q,\sigma-p)
  \lambda_{\sigma^c} \lambda_{p} \phi_{T}
  ,
  \\
  \lambda_{\rho^c}   \cartanlambda_{\rho} \wedge \eps(\sigma,\sigma^c) \phi_{\sigma^c}
  &=
  (-1)^{k+1} \eps(q,\rho-q) \eps(p,\rho-q)
  \lambda_{\rho^c} \lambda_{q} \phi_{T}
  .
 \end{align*}
 The proof is complete. 
\springerqed\end{proof}

The following combines the preceding results
and implies that the pairings are indeed Hermitian.

\begin{lemma}
  Let $r \in \bbN_{0}$, $k \in [0:n]$,
  and $\sigma, \rho \in \Sigma(k,n)$.
  Then 
  \begin{align*}
    \cartanlambda_{\sigma} \wedge \eps(\rho,\rho^c) \lambda_\rho \phi_{\rho^c}
    &=
    \cartanlambda_{\rho} \wedge \eps(\sigma,\sigma^c) \lambda_\sigma \phi_{\sigma^c}
    ,
    \\
    \lambda_{\sigma^c} \cartanlambda_{\sigma} \wedge \eps(\rho,\rho^c) \phi_{\rho^c}
    &=
    \lambda_{\rho^c} \cartanlambda_{\rho} \wedge \eps(\sigma,\sigma^c) \phi_{\sigma^c}
    .
  \end{align*}
\end{lemma}

\begin{proof}
  This follows from Lemmas~\ref{prop:pairing_internal_sparsity}--\ref{prop:pairing_internal_offdiagonal}.
\springerqed\end{proof}

Loosely speaking, the pairings are Hermitian, sparse, and have non-zero diagonal entries. 
The following two main results of this section characterize the degeneracy spaces 
of those pairings of coefficients. In particular, 
the pairings are semidefinite.

\begin{theorem} \label{prop:first_duality_pairing}
 Let $r \in \bbN_{0}$ and $k \in [0:n]$.
 Let $\omega_{\alpha\sigma}$ be a family of complex numbers indexed over
 $\alpha \in A(r,n)$ and $\sigma \in \Sigma(k,n)$.
 Then we have
 \begin{align}
  \label{math:first_duality_pairing}
  \begin{split}
   & 
   \sum_{\alpha,\beta \in A(r,n)}
   \sum_{\sigma,\rho \in \Sigma(k,n)}
   \int_{T}
   \omega_{\alpha\sigma} \lambda^{\alpha} \cartanlambda_{\sigma}
   \wedge
   \eps(\rho,\rho^{c}) 
   \overline{\omega_{\beta\rho}} \lambda^{\beta} \lambda_{\rho} \phi_{\rho^{c}}
   \\&\quad=
   (-1)^{k}
   \sum_{\theta \in \Sigma(k+1,n) } \int_{T} \lambda_{\theta} 
   \Big\vert
   \sum_{\alpha \in A(r,n)} \sum_{p \in [\theta]}
   \eps(p,\theta-p) \lambda^{\alpha} \omega_{\alpha,\theta-p}
   \Big\vert^{2}
   \phi_{T}.
  \end{split}
\end{align}
 In particular, this term is zero if and only if one of the
 equivalent conditions of Theorem~\ref{prop:coefficientequivalence_eins}
 and Lemma~\ref{prop:coefficientcondition_eins} is satisfied.
\end{theorem}

\begin{proof}
 Let us write $S(\omega)$ for the left-hand side in Equation~\eqref{math:first_duality_pairing}. 
 We can split that sum into two parts. 
 On the one hand, we have the \emph{diagonal part},
 \begin{align*}
  S_d(\omega) 
  &:=
  \sum_{ \substack{ \alpha,\beta \in A(r,n) \\ \sigma \in \Sigma(k,n) } }
\int_{T}
  \omega_{\alpha\sigma} \overline{\omega_{\beta\sigma}} 
  \lambda^{\alpha+\beta} 
  \lambda_{\sigma} 
  \cartanlambda_{\sigma} \wedge \eps(\sigma,\sigma^{c}) \phi_{\sigma^{c}}
  \\&=
  \sum_{ \substack{ \alpha,\beta \in A(r,n) \\ \sigma \in \Sigma(k,n) } }
  \int_{T}
  \omega_{\alpha\sigma} \overline{\omega_{\beta\sigma}} 
  \lambda^{\alpha+\beta} 
  \lambda_{\sigma} 
  (-1)^k \sum_{q \in [\sigma^c]} \lambda_{q} \phi_T
  \\&=
  (-1)^k 
  \sum_{ \substack{ \alpha,\beta \in A(r,n) \\ \sigma \in \Sigma(k,n) } }
  \int_{T}
  \lambda^{\alpha+\beta} 
  \omega_{\alpha\sigma} \overline{\omega_{\beta\sigma}} 
  \lambda_{\sigma} 
  \sum_{q \in [\sigma^c]} \lambda_{q} \phi_T
  ,
 \end{align*}
where we have used Lemma~\ref{prop:pairing_internal_diagonal}.
 On the other hand, for the \emph{off-diagonal part},
 \begin{align*}
  &
  S_o(\omega) 
  :=
  \sum_{ \substack{ \alpha,\beta \in A(r,n) \\ \sigma,\rho \in \Sigma(k,n) \\ \sigma \neq \rho } }
  \int_{T}
  \lambda^{\alpha+\beta} 
  \omega_{\alpha\sigma} \overline{\omega_{\beta\rho}} 
  \cartanlambda_{\sigma} \wedge \eps(\rho,\rho^{c}) \lambda_{\rho} \phi_{\rho^{c}}
  \\&
  \quad
  =
  \sum_{ \substack{ \sigma \in \Sigma(k,n) \\ \alpha, \beta \in A(r,n) \\ p \in [\sigma], \; q \in [\sigma^c] } }
  (-1)^{k+1} 
  \int_{T}
  \lambda^{\alpha+\beta} 
  \omega_{\alpha\sigma} \overline{\omega_{\beta,\sigma - p + q}} 
  \eps(p,\sigma-p) \eps(q,\sigma-p)
  \lambda_{\sigma - p + q}
  \lambda_{p} 
  \phi_{T}
  , 
 \end{align*}
 where the last equality is due to 
 Lemma~\ref{prop:pairing_internal_sparsity} and Lemma~\ref{prop:pairing_internal_offdiagonal}.
 Since by definition $S(\omega) = S_d(\omega) + S_o(\omega)$, 
 we combine that $(-1)^{k} S(\omega)$ equals
 \begin{align*}
  \sum_{ \substack{ \alpha, \beta \in A(r,n) \\ \sigma \in \Sigma(k,n) \\ q \in [\sigma^c] } }
  \int_{T} 
  \lambda_{\sigma+q}
  \lambda^{\alpha+\beta}
  \omega_{\alpha\sigma}
  \left( 
   \overline{\omega_{\beta\sigma}}
   - 
   \sum_{p\in[\sigma]} 
   \eps(p,\sigma-p) \eps(q,\sigma-p) \overline{\omega_{\beta,\sigma-p+q}}
  \right)
  \phi_T
  .
 \end{align*}
 We want to rewrite the terms in the brackets.
 Let $\alpha,\beta \in A(r,n)$ be arbitrary. 
 For any $\sigma \in \Sigma(k,n)$ and $q \in [\sigma^{c}]$
 we consider $\theta = \sigma + q$. 
 One sees that 
 \begin{align*}
    &
    \overline{\omega_{\beta,\sigma}} 
    -
    \sum_{p \in [\sigma]}
    \eps(p,\sigma-p) \eps(q,\sigma-p) \overline{\omega_{\beta,\sigma-p+q}} 
    \\&=
    \overline{\omega_{\beta,\theta-q}} 
    -
    \sum_{p \in [\sigma]}
    \eps(p,\theta-q-p) \eps(q,\theta-q-p) \overline{\omega_{\beta,\theta-p}} 
    \\&=
    \overline{\omega_{\beta,\theta-q}} 
    -
    \sum_{p \in [\sigma]}
    \eps(p,q) \eps(p,\theta-p) \eps(q,p) \eps(q,\theta-q) \overline{\omega_{\beta,\theta-p}}
    ,
 \end{align*}
 where we have used some simple combinatorial observations. 
 Together with the simple fact $\eps(p,q) \eps(q,p) = -1$, 
 we can rewrite the previous sum as  
 \begin{align*}
\overline{\omega_{\beta,\theta-q}} 
    +
    \sum_{p \in [\sigma]}
    \eps(p,\theta-p) \eps(q,\theta-q) \overline{\omega_{\beta,\theta-p}} 
=
    \eps(q,\theta-q) 
    \sum_{p \in [\theta]}
    \eps(p,\theta-p) \overline{\omega_{\beta,\theta-p}} 
    .
 \end{align*}
 Equipped with that, we can finally rewrite the terms in brackets:
 a re-indexing of the sum shows that $(-1)^{k} S(\omega)$ equals
 \begin{align*}
  &
  \sum_{ \substack{ \alpha, \beta \in A(r,n) \\ \theta \in \Sigma(k+1,n) } }
  \int_{T}
  \lambda_{\theta}
  \lambda^{\alpha+\beta}
  \eps(q,\theta-q) 
  \omega_{\alpha,\theta-q}
  \left( 
   \sum_{p\in[\theta]} 
   \eps(p,\theta-p) 
   \overline{\omega_{\beta,\theta-p}}
  \right)
  \phi_T
  .
 \end{align*}
 A sharp look at this expression reveals that the integrals 
 contain the product of a complex-valued polynomial with its adjoint, 
 multiplied by the weight factor $\lambda_{\theta}$.
 Hence we simplify the expression to 
 \begin{align*}
  &
\sum_{ \theta \in \Sigma(k+1,n) }
  \int_{T}
  \lambda_{\theta}
    \Big\vert
    \sum_{ \substack{ \alpha \in A(r,n) } }
    \sum_{p\in[\theta]} 
    \eps(p,\theta-p) 
    \lambda^{\alpha}
    \omega_{\alpha,\theta-p}
    \Big\vert^{2}
  \phi_T
  . 
 \end{align*}
 The integrand vanishes if and only if
 the conditions of Theorem~\ref{prop:coefficientequivalence_eins}
 and Lemma~\ref{prop:coefficientcondition_eins}
 are fulfilled.
 This completes the proof.
\springerqed\end{proof}

\begin{theorem} \label{prop:second_duality_pairing}
 Let $r \in \bbN_{0}$ and $k \in [0:n]$.
 Let $\omega_{\alpha\sigma}$ be a family of complex numbers indexed over
 $\alpha \in A(r,n)$ and $\sigma \in \Sigma(k,n)$.
 Then we have
 \begin{align}
  \label{math:second_duality_pairing}
  \begin{split}
   &
   \sum_{\alpha,\beta \in A(r,n)}
   \sum_{\sigma,\rho \in \Sigma(k,n)}
   \int_{T}
   \omega_{\alpha\sigma} \lambda^{\alpha} \lambda_{\sigma^c} \cartanlambda_{\sigma}
   \wedge
   \eps(\rho,\rho^{c}) 
   \overline{\omega_{\beta\rho}} \lambda^{\beta} \phi_{\rho^{c}}
   \\&\quad
   =
   (-1)^{k}
   \sum_{\theta \in \Sigma(k+1,n) }
   \int_{T} \lambda_{\theta^c} 
   \Big\vert
   \sum_{ \beta \in A(r+1,n) } \sum_{ p \in [\theta] \cap [\beta] }
   \eps(p,\theta-p) \lambda^{\beta} \omega_{\beta-p,\theta-p}
   \Big\vert^{2}
   \phi_{T}
   .
  \end{split}
\end{align}
 In particular, this term is zero if and only if one of the
 equivalent conditions of Theorem~\ref{prop:coefficientequivalence_zwei}
 and Lemma~\ref{prop:coefficientcondition_zwei}
 is satisfied.
\end{theorem}

\begin{proof}
  This works similar as in the previous proof. 
  We write $S(\omega)$ for the left-hand side in Equation~\eqref{math:second_duality_pairing}  
  and split that sum into two parts. 
  On the one hand, we have the \emph{diagonal part},
  \begin{align*}
    S_d(\omega)
    &:=
\sum_{ \substack{ \alpha,\beta \in A(r,n) \\ \sigma \in \Sigma(k,n) } }
    \int_{T}
    \omega_{\alpha\sigma} \overline{\omega_{\beta\sigma}} 
    \lambda^{\alpha+\beta} 
    \lambda_{\sigma^c} \cartanlambda_{\sigma}
    \wedge
    \eps(\sigma,\sigma^{c}) 
    \phi_{\sigma^{c}}
    \\&=
    \sum_{ \substack{ \alpha,\beta \in A(r,n) \\ \sigma \in \Sigma(k,n) } }
    \int_{T}
    \omega_{\alpha\sigma} \overline{\omega_{\beta\sigma}}
    \lambda^{\alpha+\beta} \lambda_{\sigma^c}
    (-1)^{k}
    \sum_{ q \in [\sigma^c] } \lambda_{q} \phi_T
    ,
\end{align*}
 where the last equality is due to Lemma~\ref{prop:pairing_internal_diagonal}.
 On the other hand, for the \emph{off-diagonal part},
 Lemma~\ref{prop:pairing_internal_sparsity} and Lemma~\ref{prop:pairing_internal_offdiagonal} provide that 
 \begin{align*}
  S_o(\omega) 
  &:=
  \sum_{ \substack{ \alpha,\beta \in A(r,n) \\ \sigma,\rho \in \Sigma(k,n) \\ \sigma \neq \rho } }
  \int_{T}
  \omega_{\alpha\sigma} \overline{\omega_{\beta\rho}}
  \lambda^{\alpha+\beta} 
  \lambda_{\sigma^c} 
  \cartanlambda_{\sigma}
  \wedge
  \eps(\rho,\rho^{c}) 
  \phi_{\rho^{c}}
\\&=
  \sum_{ \substack{ \sigma \in \Sigma(k,n) \\ \alpha, \beta \in A(r,n) \\ p \in [\sigma], \; q \in [\sigma^c] } }
  \int_{T}
  \omega_{\alpha\sigma} \overline{\omega_{\beta,\sigma-p+q}}
  \lambda^{\alpha+\beta} \lambda_{\sigma^c}
  (-1)^{k+1}
  \eps(p,\sigma-p) \eps(q,\sigma-p)
  \lambda_{p}
  \phi_T
  .
 \end{align*}
 Since $S(\omega) = S_d(\omega) + S_o(\omega)$, 
 we derive that $(-1)^{k} S(\omega)$ equals
 \begin{align*}
  &
\sum_{ \substack{ \alpha,\beta \in A(r,n) \\ \sigma \in \Sigma(k,n) \\ q \in [\sigma^c] } }
  \int_T
  \omega_{\alpha\sigma}
  \lambda^{\alpha+\beta}
  \lambda_{\sigma^{c}}
  \left(
   \overline{\omega_{\beta\sigma}} \lambda_{q}
   -
   \sum_{p \in [\sigma]}
   \eps(p,\sigma-p) \eps(q,\sigma-p) \overline{\omega_{\beta,\sigma-p+q}} \lambda_{p}
  \right)
  \phi_T
  .
 \end{align*}
 We want to rewrite the terms in the brackets.
 Let $\alpha,\beta \in A(r,n)$ be arbitrary. 
 For any $\sigma \in \Sigma(k,n)$ and $q \in [\sigma^{c}]$
 we consider $\theta = \sigma + q$. 
 We observe 
 \begin{align*}
    &
    \overline{\omega_{\beta,\sigma}} \lambda_{q}
    -
    \sum_{p \in [\sigma]}
    \eps(p,\sigma-p) \eps(q,\sigma-p) \overline{\omega_{\beta,\sigma-p+q}} \lambda_{p}
    \\&=
    \overline{\omega_{\beta,\theta-q}} \lambda_{q}
    -
    \sum_{p \in [\sigma]}
    \eps(p,\theta-q-p) \eps(q,\theta-q-p) \overline{\omega_{\beta,\theta-p}} \lambda_{p}
    \\&=
    \overline{\omega_{\beta,\theta-q}} \lambda_{q}
    -
    \sum_{p \in [\sigma]}
    \eps(p,q) \eps(p,\theta-p) \eps(q,p) \eps(q,\theta-q) \overline{\omega_{\beta,\theta-p}} \lambda_{p}
    ,
 \end{align*}
 where we have used some simple combinatorial observations. 
Similar as in the foregoing proof, 
 we rewrite the previous sum as 
{\fontsize{9.2pt}{\baselineskip}\selectfont
\begin{align*}
\overline{\omega_{\beta,\theta-q}} \lambda_{q}
    +
    \sum_{p \in [\sigma]}
    \eps(p,\theta-p) \eps(q,\theta-q) \overline{\omega_{\beta,\theta-p}} \lambda_{p}
=
    \eps(q,\theta-q) 
    \sum_{p \in [\theta]}
    \eps(p,\theta-p) \overline{\omega_{\beta,\theta-p}} \lambda_{p}
    .
 \end{align*}
 }
 Thus we can rewrite the terms in the bracket and see that $(-1)^{k} S(\omega)$ equals 
\begin{align*}
  &
\sum_{ \substack{ \alpha,\beta \in A(r,n) \\ \theta \in \Sigma(k+1,n) \\ q \in [\theta] } }
  \int_T
  \lambda_{\theta^{c}}
  \lambda^{\alpha+\beta}
  \eps(q,\theta-q)
  \omega_{\alpha,\theta-q}
  \lambda_{q}
  \left(
   \sum_{p \in [\theta]}
   \eps(p,\theta-p) \overline{\omega_{\beta,\theta-p}} \lambda_{p}
  \right)
  \phi_T
  .
 \end{align*}
 In analogy to the previous proof, we simplify this to 
\begin{align*}
  &
  \sum_{ \theta \in \Sigma(k+1,n) }
  \int_T
  \lambda_{\theta^{c}}
  \Big\vert
   \sum_{\alpha \in A(r,n)}
   \sum_{p \in [\theta]}
   \eps(p,\theta-p)
   \omega_{\alpha,\theta-p}
   \lambda^{\alpha}
   \lambda_{p}
  \Big\vert^{2}
  \phi_T
  \\&=
  \sum_{ \theta \in \Sigma(k+1,n) }
  \int_T
  \lambda_{\theta^{c}}
  \Big\vert
   \sum_{\beta \in A(r+1,n)}
   \sum_{p \in [\theta] \cap [\beta] }
   \eps(p,\theta-p)
   \omega_{\beta-p,\theta-p}
   \lambda^{\beta}
  \Big\vert^{2}
  \phi_T
  . \end{align*}
 The integrand vanishes if and only if
 the conditions of Theorem~\ref{prop:coefficientequivalence_zwei}
 and Lemma~\ref{prop:coefficientcondition_zwei}
 are fulfilled.
 This completes the proof.
\springerqed\end{proof}

Theorems~\ref{prop:first_duality_pairing}~and~\ref{prop:second_duality_pairing}
show that the Hermitian bilinear pairings \eqref{math:coefficientpairing:eins} and \eqref{math:coefficientpairing:zwei} 
are semidefinite:
positive or negative semidefinite depending on whether $k$ is even or odd. 
Evidently, 
the degeneracy space of the first bilinear form 
is the linear subspace of $\calP(k,r,n)$ 
spanned by those coefficient vectors 
that satisfy the conditions of Theorem~\ref{prop:coefficientequivalence_eins}
and Lemma~\ref{prop:coefficientcondition_eins}.
Analogously,
the degeneracy space of the second bilinear form 
is the linear subspace of $\calP(k,r,n)$ 
spanned by those coefficient vectors 
that satisfy the conditions of Theorem~\ref{prop:coefficientequivalence_zwei} 
and Lemma~\ref{prop:coefficientcondition_zwei}.

The main consequence is this: these bilinear forms on the coefficient spaces $\calP(k,r,n)$
describe non-degenerate bilinear pairings of isomorphic finite element spaces. 
Specifically, 
we have non-degenerate bilinear forms 
\begin{align}
 \label{math:finiteelementpairing:first}
 (\omega,\eta)
 \mapsto
 \int_{T} \omega \wedge \eta,
 \quad
 \omega \in \calP_{r}\Lambda^{k}(T),
 \quad
 \eta \in \mathring\calP^{-}_{r+k+1}\Lambda^{n-k}(T)
 ,
 \\ 
 \label{math:finiteelementpairing:second}
 (\omega,\eta)
 \mapsto
 \int_{T} \omega \wedge \eta,
 \quad
 \omega \in \mathring\calP_{r+n-k+1}\Lambda^{k}(T),
 \quad
 \eta \in \calP^{-}_{r+1}\Lambda^{n-k}(T)
 .
\end{align}

\begin{remark}
 \label{rem:duality:lit}
 Our Theorem~\ref{prop:first_duality_pairing} extends Proposition 3.7 in \cite{christiansen2013high},
 while our Theorem~\ref{prop:second_duality_pairing} states the natural
 but hitherto unpublished analogue for the second isomorphism relation. 
 Our first duality pairing is also used in Lemma~4.11 of \cite{AFW1},
 whereas our second duality pairing is utilized in Lemma~4.7 of \cite{AFW1}. 
\end{remark}

\section{Geometric Decompositions and Degrees of Freedom}
\label{sec:geometricdecompositions}

In this section we describe geometric decompositions of finite element spaces and their degrees of freedom. 
We present the main ideas within an abstract framework. 
The families of finite element spaces $\calP_{r}\Lambda^{k}(\calT)$ and $\calP_{r}^{-}\Lambda^{k}(\calT)$ are the natural applications.
This section is expositional and shows how the non-degeneracy of the duality pairings in the previous section
leads to a direct construction of the degrees of freedom. 

Throughout this section, we let $\calT$ be a collection of simplices satisfying the following conditions:
(i) for every $T \in \calT$ and every subsimplex $F \subseteq T$ we have $F \in \calT$, 
(ii) for every two $T, T' \in \calT$ we either have $T \cap T' = \emptyset$ or $T \cap T' \in \calT$, 
(iii) we have $\dim T \leq n$ for every $T \in \calT$.

We assume to be given $X^{k}(T) \subseteq \Lambda^{k}(T)$ 
for each cell $T \in \calT$ such that 
for every $F,T \in \calT$ with $F \subseteq T$ we have the surjectivity condition 
$\trace_{T,F} X^{k}(T) = X^{k}(F)$.
We write $\mathring X^{k}(T)$ for the subspace of $k$-forms
with vanishing boundary traces:
\begin{align}
 \label{math:definitiontracefree}
 \mathring X^{k}(T)
 =
 \left\{\; \omega \in X^{k}(T) \suchthat* \forall F \in \calT, F \subsetneq T : \trace_{T,F} \omega = 0 \;\right\}.
\end{align}
Let us abbreviate $X^{k}_{-1}(\calT) := \bigoplus_{T \in \calT, \dim T = n} X^{k}(T)$
for the direct sum of vector spaces associated to the $n$-simplices.
We say that $\omega \in X^{k}_{-1}(\calT)$ is \emph{single-valued}
if for all $n$-dimensional simplices $T,T' \in \calT$ with non-empty intersection $F = T \cap T'$
we have $\trace_{T,F} \omega_{T} = \trace_{T',F} \omega_{T'}$. 
The set $X^{k}(\calT)$ of single-valued members of $X^{k}_{-1}(\calT)$ is a subspace.
The \emph{global trace operators} 
\begin{align}
 \label{math:globaltraceoperator}
 \Trace_{\calT,F} : X^{k}(\calT) \rightarrow X^k(F)
\end{align}
onto any $F \in \calT$ are defined in the obvious way. 
We remark that this notion of \emph{single-valued} agrees with the definition 
in \cite{AFWgeodecomp}.
Note that for every single-valued $\omega \in X^{k}(\calT)$
we have $\Trace_{\calT,F} \omega = 0$ if $\dim F < k$.

\begin{remark}
 $X^{k}(\calT)$ 
 captures the idea of a \emph{conforming} finite element space.
For example, we have the full spaces of barycentric polynomial differential forms,
 where $X^{k}(T) = \calP_{r}\Lambda^{k}(T)$ 
 and $\mathring X^{k}(T) = \mathring\calP_{r}\Lambda^{k}(T)$ 
 for each $T \in \calT$.
 Here, $\calP_{r}\Lambda^{k}(\calT) := X^{k}(\calT)$ is common notation. 
Another example are the spaces of higher order Whitney forms,
 where $X^{k}(T) = \calP_{r}^{-}\Lambda^{k}(T)$ 
 and $\mathring X^{k}(T) = \mathring\calP_{r}^{-}\Lambda^{k}(T)$ 
 for each $T \in \calT$.
 In this case, $\calP_{r}^{-}\Lambda^{k}(\calT) := X^{k}(\calT)$ is common notation. 
\end{remark}

Our abstract framework requires the existence of extension operators. 
For all $F, T \in \calT$ with $F \subseteq T$ we assume to have a linear mapping 
\begin{align*} 
 \ext_{F,T} : \mathring X^{k}(F) \rightarrow X^{k}(T).
\end{align*}
We assume that these are generalized inverses of the trace operators,
\begin{align}
 \label{math:extensionisrightinverse}
 \trace_{T,F} \ext_{F,T} \omega = \omega, 
 \quad
 \omega \in \mathring X^{k}(F),
\end{align}
and satisfy the following two conditions:
on the one hand, we require that extensions to different simplices have the same trace on common subsimplices,
\begin{align}
 \label{math:extensionproperty_uno}
 \ext_{F,G} \omega = \trace_{T,G} \ext_{F,T} \omega,
 \quad
 \omega \in \mathring X^{k}(F),
 \quad
 F \subseteq G \subseteq T,
 \quad
 F, G, T \in \calT,
\end{align}
and on the other hand, we require that the extension 
has zero trace on all simplices of $\calT$ that do not contain the original simplex,
\begin{align}
 \label{math:extensionproperty_due}
 \trace_{T,G} \ext_{F,T} \omega = 0,
 \quad 
 \omega \in \mathring X^{k}(F),
 \quad
 F,G \subseteq T,
 \quad
 F \nsubseteq G,
 \quad
 F, G, T \in \calT.
\end{align}
Under these assumptions, 
we easily verify that the \emph{global extension operators} 
\begin{align}
 \label{math:globalextensionoperator}
 \Ext_{F,\calT} : \mathring X^{k}(F) \rightarrow X^{k}(\calT),
 \quad
 \omega_F \mapsto \sum_{ \substack{ F, T \in \calT \\ F \subseteq T, \; \dim T = n } } \ext_{F,T} \omega_F,
\end{align}
are well-defined. To every trace in $\mathring X^{k}(F)$ over some simplex $F \in \calT$ 
we thus associate a single-valued differential form with that trace and that is localized around $F$. 

\begin{remark}
  We show two examples of extension operators satisfying the abstract assumptions.
  The two families of extension operators discussed in Subsection~\ref{subsec:finiteelementspaces:extension}, 
  \begin{align}
   \ext^{k,r}_{F,T} : \mathring\calP_{r}\Lambda^{k}(F) \rightarrow \calP_{r}\Lambda^{k}(T),
   \quad 
   \ext^{k,r,-}_{F,T} : \mathring\calP_{r}^{-}\Lambda^{k}(F) \rightarrow \calP_{r}^{-}\Lambda^{k}(T),
  \end{align}
  satisfy the required conditions of this section, 
  and thus lead to geometric decompositions 
  of $\calP_{r}\Lambda^{k}(\calT)$ and $\calP_{r}^{-}\Lambda^{k}(\calT)$, respectively. 
\end{remark}

\begin{theorem} \label{prop:geometricdecomposition}
 Suppose that $\omega \in X^{k}(\calT)$.
 Then there exist unique $\omega_F \in \mathring X^{k}(F)$ for every $F \in \calT$
 such that
 \begin{align}
  \label{math:geometricdecomposition}
  \omega = \sum_{ F \in \calT} \Ext_{F,\calT} \omega_F.
 \end{align}
\end{theorem}

\begin{proof}
 Let $\omega \in X^k(\calT)$.
 We prove the theorem via recursion.
 Recall $\Trace_{\calT,F} \omega = 0$ whenever $F \in \calT$ with $\dim(F) < k$.
 Let $\omega_V := \Trace_{\calT,V} \omega \in \mathring X^{k}(V)$
 for every vertex $V \in \calT$ of the simplicial complex
 and let 
 \begin{align*}
  \omega^{(0)} := \sum_{ V \in \calT, \; \dim V = 0 } \Ext_{V,\calT} \omega_V
  .
 \end{align*}
 Then
 $\Trace_{\calT,V}( \omega - \omega^{(0)} ) = 0$ 
 for every $0$-dimensional $V \in \calT$.
 
 Now assume that for some $m \in [0:n-1]$ the following holds:
 for every $F \in \calT$ of dimension at most $m$
 there exists $\omega_F \in \mathring X^{k}(F)$ such that,
 letting 
 \begin{align*}
  \omega^{(m)} := \sum_{ F \in \calT, \; \dim F \leq m } \Ext_{F,\calT} \omega_F,
 \end{align*}
 we have $\Trace_{\calT,F}( \omega - \omega^{(m)} ) = 0$
 for every $F \in \calT$ of dimension at most $m$.
 For every $F \in \calT$ of dimension $m+1$, 
 we set $\omega_F := \Trace_{\calT,F} ( \omega - \omega^{(m)} ) \in \mathring X^{k}(F)$. 
Letting 
 \begin{align*}
  \omega^{(m+1)} := \sum_{ F \in \calT, \; \dim F \leq m+1 } \Ext_{F,\calT} \omega_F,
 \end{align*}
 it follows that $\Trace_{\calT,F}( \omega - \omega^{(m+1)} ) = 0$
 for every $F \in \calT$ of dimension at most $m+1$.
 
 Iterating this construction produces $\omega_{F} \in \mathring X^{k}(F)$
 for every $F \in \calT$ such that 
 \begin{align*}
  \omega - \sum_{ F \in \calT } \Ext_{F,\calT} \omega_F
 \end{align*}
 has vanishing trace on every $F \in \calT$. 
 Thus \eqref{math:geometricdecomposition} follows. 
 
 To show uniqueness of the construction it suffices to show 
 that $\omega = 0$ only if $\omega_F = 0$ for all $F \in \calT$.
 Suppose that there exists $F \in \calT$ such that $\omega_F \neq 0$. 
 Without loss of generality, 
 $\omega_f = 0$ for all $f \in \calT$ with $\dim f \leq \dim F$.
 Then $\Trace_{\calT,F} \omega = \omega_F$.
 Since $\omega_F \neq 0$ we must have $\omega \neq 0$, 
 which had to be shown and implies uniqueness of the decomposition. 
\springerqed\end{proof}

\begin{remark}
 Any basis of $\mathring X^{k}(F)$ gives a basis of $\Ext_{F,\calT} \mathring X^{k}(F)$. 
 In the light of the geometric decomposition \eqref{math:geometricdecomposition},
 we see that choosing a basis for each space $\mathring X^{k}(F)$ 
 gives rise to a geometrically decomposed basis for $X^{k}(\calT)$. 
Our exposition, and many of the references cited, emphasize the definition
 of the finite element basis directly on the physical cells.
 Transformation from reference cells is another widespread formalism for
 the generation of finite element bases \cite{ern2021finite}.
\end{remark}

We finish this section with a discussion of the degrees of freedom. 
For each $F \in \calT$ of dimension $\dim F = m$ we define the spaces of functionals 
\begin{align*}
  W_{k,r}(F)
  &:=
  \left\{
    \phi \in \calP_{r}\Lambda^{k}(\calT)^{\ast}
    \suchthat*
    \exists \eta \in \calP_{r+k-m}^{-}\Lambda^{m - k}(F):
    \phi(\cdot)
    =
    \int_{F} \Trace \cdot \wedge \eta 
  \right\},
  \\
  W_{k,r}^{-}(F)
  &:=
  \left\{
    \phi \in \calP_{r}^{-}\Lambda^{k}(\calT)^{\ast} 
    \suchthat*
    \exists \eta \in \calP_{r+k-m-1}\Lambda^{m - k}(F):
    \phi(\cdot)
    =
    \int_{F} \Trace \cdot \wedge \eta 
  \right\}.
\end{align*}
We have isomorphisms $W_{k,r}(F) \simeq \mathring\calP_{r}\Lambda^{k}(F)^{\ast}$
and $W_{k,r}^{-}(F) \simeq \mathring\calP_{r}^{-}\Lambda^{k}(F)^{\ast}$,
as is evident from considering the pairings \eqref{math:finiteelementpairing:first} 
and \eqref{math:finiteelementpairing:second}.
Notice that $W_{k,r}(F)$ and $W_{k,r}^{-}(F)$ are trivial when $\dim(F) < k$.
We define the spaces 
\begin{align} \label{math:globaldofspace}
  W_{k,r}(\calT)
  :=
  \sum_{ F \in \calT, \; \dim F \geq k } W_{k,r}(F)
  ,
  \quad 
  W_{k,r}^{-}(\calT)
  :=
  \sum_{ F \in \calT, \; \dim F \geq k } W_{k,r}^{-}(F)
  .
\end{align}
These are spaces of functionals over conforming finite element spaces. 
In fact, they are direct sums and already constitute the entire degrees of freedom of the finite element spaces.

\begin{theorem}
 For $r \in \bbN_{0}$ and $k \in [0:n]$,
 the sums \eqref{math:globaldofspace} are direct and we have 
 \begin{align}
  \calP_{r}\Lambda^{k}(\calT)^{\ast}
  = 
  W_{k,r}(\calT)
  ,
  \quad 
  \calP_{r}^{-}\Lambda^{k}(\calT)^{\ast}
  = 
  W_{k,r}^{-}(\calT)
  .
 \end{align}
\end{theorem}

\begin{proof}
 We state the proof for the first identity;
 the proof for the second identity is completely analogous. 
 Recall that $W_{k,r}(\calT) \subseteq \calP_{r}\Lambda^{k}(\calT)^{\ast}$.
 To show the reverse inclusion, 
 it suffices to show that for any $\omega \in \calP_{r}\Lambda^{k}(\calT)$
 satisfying $\eta(\omega) = 0$ for all $\eta \in W_{k,r}(\calT)$, 
 we already have $\omega = 0$.
 Let $\omega$ have that property. 
 
 Recall that by Theorem~\ref{prop:geometricdecomposition} 
 we have $\omega = \sum_{F \in \calT} \Ext_{F,\calT} \omega_{F}$
 with unique $\omega_{F} \in \mathring\calP_{r}\Lambda^{k}(F)$ for each $F \in \calT$. 
Let $m \in \bbN_{0}$
 such that $\omega_{F} = 0$ for every $F \in \calT$ with $\dim F < m$.
 Let $F \in \calT$ with $\dim F = m$. 
 By assumption, $\Trace_{\calT,F} \omega = \omega_{F}$,
 and since $\eta(\omega) = 0$ for all $\eta \in W_{k,r}(F)$,
 we have $\omega_{F} = 0$. 
 Iterating this, starting with $m=k$ and increasing $m$, we obtain $\omega = 0$.
 That establishes $\calP_{r}\Lambda^{k}(\calT)^{\ast} = W_{k,r}(\calT)$.

 We want to show that the sum defining $W_{k,r}(\calT)$ is direct. 
 Let $\eta \in W_{k,r}(\calT)$ be given by $\eta = \sum_{F \in \calT, \; \dim F \geq k} \eta_{F}$,
 where $\eta_{F} \in W_{k,r}(F)$ are not all zero. 
 We need to show $\eta \neq 0$.
 For that it suffices to find $\omega \in W_{k,r}(\calT)$ such that $\eta(\omega) \neq 0$.
 Note that there exists $F \in \calT$ such that $\eta_F \neq 0$ but $\eta_T = 0$
 for all $T \in \calT$ with $\dim(T) > \dim(F)$.
 We fix such $F$ and pick $\omega_F \in \mathring\calP_{r}\Lambda^{k}(F)$ 
 such that $\eta_F(\omega_F) \neq 0$.
 Then $\omega := \Ext_{F,\calT} \omega_{F}$
 satisfies $\eta(\omega) = \eta_F(\omega_F)$,
 and the desired result follows. 
\springerqed\end{proof}

\subsection*{Acknowledgements}
The author enjoyed helpful discussions with Douglas N.\ Arnold and Snorre H.\ Christiansen.
The referees' careful reading and valuable comments are appreciated.

\end{document}